\documentclass[letterpaper, 10 pt, conference]{ieeeconf}
\IEEEoverridecommandlockouts 
\usepackage{hyphenat}
\usepackage{longtable}
\usepackage{lettrine}
\usepackage{algorithm}
\usepackage{amsmath}
\usepackage{amssymb}
\usepackage[utf8]{inputenc}
\inputencoding{latin1}
\usepackage{array}
\usepackage[mathscr]{euscript}
\usepackage{float}
\usepackage{cancel}
\usepackage{mdframed}
\usepackage{pifont}
\usepackage{epigraph}
\usepackage[most]{tcolorbox}
\usepackage{wrapfig}
\usepackage{hyperref}
\usepackage{color}
\usepackage{lipsum}
\usepackage{graphicx}      
\usepackage{arydshln}

\definecolor{myred}{RGB}{51,0,0}
\definecolor{myblue}{RGB}{0,34,142}

\newtheorem{Theorem}{\textbf{Theorem}}
\newtheorem{Lemma}{\textbf{Lemma}}

\newtheorem{Definition}{\textbf{Definition}}

\newtheorem{Example}{\textbf{Example}}

\newtheorem{Remark}{\textbf{Remark}}
\newtheorem{Notation}{\textbf{Notation}}

\newcommand{\NX}{\nLPV_\mathrm{x}}
\newcommand{\NY}{\nLPV_\mathrm{y}}
\newcommand{\NU}{\nLPV_\mathrm{u}}
\newcommand{\NP}{\nLPV_\mathrm{p}}
\newcommand{\NPSI}{\nLPV_\mathrm{\Psi}}

\definecolor{orange}{rgb}{0.1,0.5,1}
\definecolor{Lightgrey}{rgb}{0.95,0.95,0.95}
\definecolor{VeryLightgrey}{rgb}{0.98,0.98,0.98}

\newcommand{\Diag}{\mathrm{diag}}

\newcommand{\LFTA}{\mathsf{A}}
\newcommand{\LFTB}{\mathsf{B}}
\newcommand{\LFTC}{\mathsf{C}}
\newcommand{\LFTD}{\mathsf{D}}

\newcommand{\LFTM}{\mathcal{M}}

\newcommand{\LFTMatrix}{\mathsf{M}}

\newcommand{\ie}{\emph{i.e.}}

\newcommand{\nLFT}{\mathrm{r}}

\newcommand{\LPVA}{\mathscr{A}}
\newcommand{\LPVB}{\mathscr{B}}
\newcommand{\LPVC}{\mathscr{C}}
\newcommand{\LPVD}{\mathscr{D}}

\newcommand{\LPVS}{\mathscr{S}}
\newcommand{\LPVP}{\mathscr{P}}
\newcommand{\LPVX}{\mathscr{X}}

\newcommand{\nLPV}{{n}}

\newcommand{\I}{\mathrm{I}}

\newcommand{\In}[1]{{\mathrm{I}}_{#1}}

\newcommand{\inrange}[2]{\mathbb{I}_{#1}^{#2}}

\title{
   On the equivalence between functionally affine LPV  state-space representations and LFT models.
}
\author{Mih\'aly Petreczky$^{1}$ and Ziad Alkhoury$^{2}$  and Guillaume Merc\`ere$^{3}$
\thanks{$^{1}$ Univ. Lille, CNRS, Centrale Lille, UMR 9189 CRIStAL, F-59000 Lille, France
 {\tt\small mihaly.petreczky@centralelille.fr}}%
\thanks{$^{2}$ Autoliv Electronics, Cergy, France {\tt\small ziad.alkhoury@autoliv.com}}
\thanks{$^{3}$ Universit\'e de Poitiers, ISAE-ENSMA, LIAS, Poitiers, France,
   {\tt\small guillaume.mercere@univ-poitiers.fr}}%
}

\begin{document}

\maketitle
\thispagestyle{empty}

\begin{abstract}
\textcolor{red}{We} propose a transformation algorithm for a class of Linear Parameter-Varying (LPV) systems with 
functional affine dependence on parameters, where the system
 matrices depend affinely on nonlinear functions of the scheduling variable, 
 into Linear Fractional Transformation (LFT) systems. The transformation preserves input-output behavior and minimality,
 \textcolor{red}{and the uncertainity block of the resulting LFT system is linear in the scheduling variables of the LPV system.}
\end{abstract}

\section{Introduction}
Linear Fractional Transformation (LFT) \cite{Zhou1996} and Linear Parameter-Varying (LPV) 
systems \cite{Briat2015,Toth2010SpringerBook,Rugh2000} are crucial 
for handling uncertainties and nonlinearities.
LFT effectively represents uncertain systems, while LPV models systems with time-varying parameters. 
Transforming LPV systems into LFT form, where the uncertainty block depends on the scheduling signal, is 
common \cite{Briat2015,Scherer2009}. 
\textcolor{red}{For LPV systems with affine dependence on the scheduling signal, this transformation is straightforward, but for the general case, it becomes more complex.}

We propose \textcolor{red}{an algorithm for transforming} a class of \textcolor{red}{discrete-time} LPV systems with a 
\emph{functional affine dependence on the parameters} (FALPV), where the matrices depend affinely on
 \textcolor{red}{\emph{recognizable} nonlinear functions of the scheduling variable, to LFT systems.
 Recognizable functions are functions which can be represented by LFT models, and they include polynomial and, after scaling, rational functions.}
    \textcolor{red}{The uncertainty
    block of the resulting LFT system is linear in the scheduling variables of the FALPV and has the same number of uncertainty blocks 
    as the number of scheduling variables.
     The transformation preserves input-output behavior and minimality. }
     \textcolor{red}{In particular, the LFTs arising from two minimal and input-output equivalent FALPVs are isomorphic.}

\textbf{Motivation}
\textcolor{red}{The proposed transformation is useful, because}
the resulting uncertainty block being linear in the scheduling variable simplifies controller synthesis and reduces conservatism. 
While FALPVs can be interpreted as ALPVs after a non-linear transformation of the scheduling variable, 
this approach results in an LFT with a non-linear dependence on the scheduling variable and possibly more uncertainty operators than the 
number of scheduling variables. 
\textcolor{red}{For instance, for LPVs with a polynomial dependence the uncertainity block would contain
monomials of the scheduling variables.}
This makes it harder to bound the uncertainty block and increases the conservatism of the controller synthesis.
Preservation of minimality and isomorphism 
makes controller synthesis independent of the modeling approach. 
This is significant as LPV models
often originate from system identification \cite{Toth2010SpringerBook,CoxTothSubspace}, and
different identification methods may yield different LPV models that are approximately input-output equivalent and minimal. 
However, naive transformations of \textcolor{red}{minimal LPV systems to LFTs can result in non-isomorphic LFTs,}
 leading to different achievable closed-loop performances \cite{Alkhoury2016}. 
 Our results imply that LFTs derived from input-output equivalent minimal LPV systems are
 minimal, input-output equivalent \textcolor{red}{and isomorphic.}
This \textcolor{red}{suggests} that there is  no advantage in identifying LFTs directly \cite{CL08, Vizer2013b} versus LPV models.
\textcolor{red}{In this paper we concentrate on discrete-time LPV systems due to their simplicity and relevance for system identification.
 We conjecture that the results could be extended to the continuous-time case.}

\textbf{Related work}
 The idea of transforming LPV systems to LFT for control design is a standard one \cite{Briat2015,Rugh2000,Scherer2009}, the structural properties of this transformation, 
such as the preservation of minimality and input-output equivalence were not investigated before for FALPV systems.
In deriving the results  we use realization theory of ALPV systems 
\cite{Petreczky2016}, and \textcolor{red}{of} 
 LFTs (viewed as multidimensional systems) \cite{Beck1999,Ball2008}.
 \textcolor{red}{For LPV  systems with affine dependence,}  the preservation of minimality and input-output equivalence for all uncertainty blocks
 was investigated in \cite{Alkhoury2016}. The results of this paper are extension of \cite{Alkhoury2016}
 to FALPV case.

\section{Definitions and problem formulation} \label{sec:definitions}

\textbf{Notation} 
 $\mathbb{N}$ denotes the set of all natural numbers,
  $t$ denotes the discrete-time instant, \ie, $t \in \mathbb{N}$ as well. 
  We denote by $\inrange{i}{j}$ the set $\{i, i+1 ,\ldots, j-1, j\}$, where $i,j \in \mathbb{N}, j>i$. 
  The $n \times n$ identity matrix is denoted by $\In{n}$.
  \textcolor{red}{Let $O_{k,l}$ be  the $k \times l$ zero matrix.}
   When its dimension is clear from the context, the identity matrix is denoted by $\I$,
   and the zero matrix by $0$.

   For a  set $X$ denote by $X^{\mathbb{N}}$ the set of functions
   from $\mathbb{N}$ to $X$. Denote by $l_2(\mathbb{R}^{n})$ the space square summable sequences from $\mathbb{R}^{n}$, and
   by $\mathcal{L}(l_2(\mathbb{R}^{n}))$
   the space of all bounded linear operators on $l_2(\mathbb{R}^{n})$.
  If $A,B$ linear operators, then $AB$ denotes their composition.
  If $s \in (\mathbb{R}^m)^{\mathbb{N}}$, then \textcolor{red}{the $i$th component $s_i \in (\mathbb{R})^{\mathbb{N}}$, $i \in \mathbb{I}_1^d$ 
  of $s$} is such that $s(t)=(s_1(t),\ldots,s_{m}(t))^T$, $t \in \mathbb{N}$.
  If $A$ is an $n \times m$ matrix, $(i,j)$th entry $A_{i,j}$ of which is a linear operator 
  from a subspace $Z \subseteq (\mathbb{R})^{\mathbb{N}}$ to  $(\mathbb{R})^{\mathbb{N}}$, then 
  $A$ is \textcolor{red}{an} operator from  $Z^m$ to  $(\mathbb{R}^p)^{\mathbb{N}}$, 
  where $Z^m$ is the subspace of all $s \in (\mathbb{R}^m)^{\mathbb{N}}$ such that the components $s_i$ of $s$ belong to $Z$, 
  $i \in \mathbb{I}_1^m$, and the $l$th component of $As$ is $\sum_{j=1}^m A_{l,j}s_j$.
  Thus, $\delta \I_m$ denotes the $m \times m$ \textcolor{red}{diagonal} 
  matrix of operators $\Diag(\delta,\ldots,\delta)$.

  For a finite set $\mathcal{X}$ let
  $\mathcal{X}^*$ be the set of all finite sequences  of elements of $\mathcal{X}$. 
  An element $w \in \mathcal{X}^*$ of length $|w|=n \in \mathbb{N}$, 
  is a sequence of the form $x_1 x_2 \dots x_n$, where $x_i \in \mathcal{X}, \forall i=1,2,...,n$. 
  Denote $\epsilon$ the empty sequence where $|\epsilon|=0$, and denote by $vw$ the concatenation of two
  sequence $w,v \in \mathcal{X}^{*}$.
  Thus, $(\mathbb{I}_i^j)^{*}$ denotes the set of all finite sequences  of $\inrange{i}{j}$. 



\subsection{Functional-Affine LPV (FALPV) Models}\label{sec:ALPV}
\begin{color}{red}
Let $\mathbb{P} \subseteq \mathbb{R}^{\NP}$  be the set of scheduling parameters, 
and let us consider a \emph{known function}
$\psi:\mathbb{P} \rightarrow \mathbb{R}^{\NPSI}$. Let us denote by $\psi_i$ the $i$th component of $\psi$, i.e.,
$\psi(\mathbf{p})=[\psi_1(\mathbf{p}),\ldots,\psi_{\NPSI}(\mathbf{p})]^{\top}$, $\forall \mathbf{p} \in \mathbb{P}$.
\end{color}
A \textcolor{red}{Functional Affine Linear Parameter-Varying (FALPV) model $\Sigma$ w.r.t $\psi$
  is a system}
\begin{equation}
\label{equ:FALPVSystem}
\Sigma\left\{
\begin{array}{lcl}
x(t+1) &=&  \LPVA(p(t)) x(t) + \LPVB(p(t)) u(t) , \\
y(t) &=& \LPVC(p(t)) x(t) + \LPVD(p(t)) u(t).
\end{array}
\right. 
\end{equation}
where $x(t) \in \mathbb{X}=\mathbb{R}^{n_\mathrm{x}}$ is the state at time $t$, $y(t) \in \mathbb{Y}=\mathbb{R}^{n_\mathrm{y}}$ 
\textcolor{red}{is} the output at time $t$, $u(t) \in \mathbb{U} = \mathbb{R}^{n_\mathrm{u}}$ represents the input at time $t$, while
 $p(t) \in \mathbb{P} \subseteq \mathbb{R}^{n_\mathrm{p}}$ is the scheduling signal at time $t$.
 The matrix \textcolor{red}{valued} functions $\LPVA$, $\LPVB$, $\LPVC$ and $\LPVD$ are assumed to be affine functions of \textcolor{red}{$\psi(p(t))$}.
\begin{equation*}
	\begin{split}
	 & \begin{bmatrix} \LPVA(p(t))\!\!\! &\!\!\! \LPVB(p(t)) \\ 
    \LPVC(p(t))\!\!\! &\!\!\! \LPVD(p(t)) \end{bmatrix} \!\!=\!\!
     \begin{bmatrix} \LPVA_0\!\! &\!\! \LPVB_0 \\ \LPVC_0\!\! &\!\! \LPVD_0 \end{bmatrix} + 
     \sum\limits_{i=1}^{\textcolor{red}{\NPSI}} \begin{bmatrix} \LPVA_i\!\! &\!\! \LPVB_i \\ \LPVC_i\!\! &\!\! \LPVD_i \end{bmatrix} \psi_i(p(t)) \\
  \end{split}
\end{equation*}
for constant matrices 
$\LPVA_i \in \mathbb{R}^{n_{\mathrm x} \times n_{\mathrm x}}$, $\LPVB_i  \in \mathbb{R}^{n_{\mathrm x} \times n_{\mathrm u}}$,
$\LPVC_i \in \mathbb{R}^{n_{\mathrm y} \times n_{\mathrm x}}$, 
$\LPVD_i \in \mathbb{R}^{n_{\mathrm y} \times n_{\mathrm u}}$, $i \in \mathbb{I}_1^{\NPSI}$, 
Without loss of generality we assume that $\mathbb{P}=[-1,1]^{\NP}$, and
the \textcolor{red}{\emph{functions $1,\psi_1, \ldots, \psi_{\NPSI}$ are linearly independent},
\textcolor{red}{i.e. no $\psi_i$ can be expressed as a linear combination of $\{1\} \cup \{\psi_j\}_{j=1, i \ne j}^{\NPSI}$, $i \in \inrange{1}{\NPSI}$}.}
\textcolor{red}{If the functions $\{1\} \cup \{\psi_i\}_{i=1}^{\NP}$} are not \textcolor{red}{linearly} independent, we can choose a linear basis from them
and express the system matrices using only this \textcolor{red}{linear} basis. The scheduling variables are usually assumed 
to be bounded, and we can always scale them to $[-1,1]^{\NP}$.

\textcolor{red}{Note that $\psi$ is not necessarily a polynomial/affine function. In principle, for any LPV system when can choose
a function $\psi$  (formed by the entries of the matrices of the LPV) for which we can rewrite that LPV as an FALPV.
}

 The \emph{dimension of $\Sigma$}, denoted by \emph{$\dim(\Sigma)$} 
is the dimension $\NX$ of its state-space.
By a solution  of $\Sigma$ we mean a tuple of trajectories $(x,y,u,p)\in(\mathcal{X},\mathcal{Y},\mathcal{U},\mathcal{P})$ satisfying 
\eqref{equ:FALPVSystem} for all $t \in \mathbb{N}$, where
$\mathcal{X}=\mathbb{X}^\mathbb{N}, \mathcal{Y}=\mathbb{Y}^{\mathbb{N}}, 
\mathcal{U}=\mathbb{U}^{\mathbb{N}},\mathcal{P}=\mathbb{P}^\mathbb{N}$. 
Define the \emph{input-output map}
	\( Y_{\Sigma} :   \mathcal{U} \times \mathcal{P}  \rightarrow \mathcal{Y} \)
of $\Sigma$ 
such that for any $(x,y,u,p) \in \mathcal{X} \times \mathcal{Y} \times \mathcal{U} \times \mathcal{P}$,
$y=Y_{\Sigma}(u,p)$ holds if and only if
$(x,y,u,p)$ is a solution of $\Sigma$ and $x(0)=0$.
Two FALPV models are said to be \emph{input-output equivalent}, if their input-output maps coincide. 
 An FALPV model $\Sigma$ is said to be
\emph{a minimal}, if for any FALPV model $\hat{\Sigma}$ which is input-output equivalent to
 $\Sigma$, $\dim(\Sigma) \le \dim(\hat{\Sigma})$. 
We assumed zero initial state for simplicity, the results can be extended to to the case of non-zero initial state.
From \cite{Petreczky2016}, 
\footnote{\textcolor{red}{We apply  \cite{Petreczky2016} with the new scheduling variable $\psi(p)$}.}
we have the following result:
\textbf{(1)} A FALPV is minimal if and only if it is extended reachability and observability matrices are full row and column rank, respectively. 
\textbf{(2)}	Any two minimal FALPVs which are input-output equivalent are isomorphic 
\textbf{(3)} Any FALPV can be transformed to an input-output equivalent minimal FALPV.

\subsection{Linear Fractional Representation}\label{sec:LFT}

\begin{wrapfigure}{r}{0.2\textwidth}
  \centering
  \scalebox{0.75}{
  \begin{picture}(100,50)(0,0)
  \linethickness{0.3mm}
  \put(56,36){\line(1,0){44}}
  \put(56,6){\line(0,1){30}}
  \put(100,6){\line(0,1){30}}
  \put(56,6){\line(1,0){44}}
  \linethickness{0.3mm}
  \put(68,68){\line(1,0){20}}
  \put(68,50){\line(0,1){18}}
  \put(88,50){\line(0,1){18}}
  \put(68,50){\line(1,0){20}}
  \put(78,21){\makebox(0,0)[cc]{$\LFTMatrix$}}
  \put(78,60){\makebox(0,0)[cc]{$\Delta$}}
  \linethickness{0.3mm}
  \put(44,28){\line(1,0){12}}
  \linethickness{0.3mm}
  \put(44,28){\line(0,1){32}}
  \linethickness{0.3mm}
  \linethickness{0.3mm}
  \put(44,60){\line(1,0){24}}
  \put(68,60){\vector(1,0){0.12}}
  \linethickness{0.3mm}
  \put(88,60){\line(1,0){24}}
  \linethickness{0.3mm}
  \put(112,28){\line(0,1){32}}
  \linethickness{0.3mm}
  \put(100,28){\line(1,0){12}}
  \put(100,28){\vector(-1,0){0.12}}
  \linethickness{0.3mm}
  \put(100,12){\line(1,0){36}}
  \put(100,12){\vector(-1,0){0.12}}
  \linethickness{0.3mm}
  \put(18,12){\line(1,0){38}}
  \put(18,12){\vector(-1,0){0.12}}
  \put(124,19){\makebox(0,0)[cc]{$u$}}
  \put(33,20){\makebox(0,0)[cc]{$y$}}
  \put(33,36){\makebox(0,0)[cc]{$z$}}
  \put(124,34){\makebox(0,0)[cc]{$w$}}
  \end{picture}
  }
  \caption[\footnotesize General LFT]{\footnotesize General LFT}
  \label{fig:lft_figure}
\end{wrapfigure}
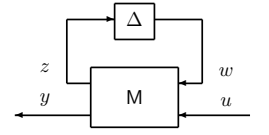
A Linear Fractional Transformation Representation (LFT) \footnote{\textcolor{red}{Note that in the literature, 
 sometimes the more logical expression of Linear Fractional Representation (LFR) is used. We prefer to use the term LFT, as it is more common.}}
is a feedback interconnection of an LTI system $(\LFTA,\LFTB,\LFTC,\LFTD)$ 
 with an uncertainty \textcolor{red}{(possibly dynamic)} operator $\Delta$ on \textcolor{red}{discrete-time signals, of the form}
 \begin{equation}
\label{delt:struc:def}
\begin{split}
  & \Delta = \Diag[\delta_1 \In{\nLFT_1}, \cdots, \delta_d \In{\nLFT_d}], \\
  \end{split}
 \end{equation}
  where $\delta_i: (\mathbb{R})^{\mathbb{N}} \rightarrow (\mathbb{R})^{N}$ , 
  $\forall i \in \inrange{1}{d}$, are linear operators on \textcolor{red}{discrete-time signals.}

Formally, an LFT is a tuple 
	\begin{equation}
	\label{LFTform}
	\LFTM=(\LFTMatrix,\Delta)=(p,m,d, \{\nLFT_{i}\}_{i=1}^{d},\LFTA,\LFTB,\LFTC,\LFTD),
	\end{equation}
	where $p,m,d, \nLFT_{i}$, $i \in \inrange{1}{d}$ are positive integers, 
  and  $\LFTA \in \mathbb{R}^{\nLFT \times \nLFT}$, $\LFTB \in \mathbb{R}^{\nLFT \times m}$,
   $\LFTC \in \mathbb{R}^{p \times \nLFT}$ and $\LFTD \in \mathbb{R}^{p \times m}$, $\nLFT=\sum_{i=1}^{d} \nLFT_{i}$. 


A \emph{solution} of $\LFTM$ for $\Delta$ of the form \eqref{delt:struc:def} is a tuple $(z,y,u)$ 
such that $z,w \in (\mathbb{R}^{\nLFT})^{\mathbb{N}}$,
$y \in (\mathbb{R}^{p})^{\mathbb{N}}$, $u \in  (\mathbb{R}^{m})^{\mathbb{N}}$,
such that $z=Aw+Bu$, $y=Cw+Du$ and  $w=\Delta z$

If we view $\LFTA$ as a linear operator on $(\mathbb{R})^{\mathbb{N}}$ applied element-wise, i.e.,
$(Az)(t)=Az(t)$, $t \in \mathbb{N}$, and the linear operator $I-A\Delta$ is invertible, when restricted to a subspace
containing $Bu$,  then the relationship between the input and the output can be expressed
$y=(\LFTM \star \Delta)u$, where 
$\LFTM \star \Delta:=\LFTD+ \LFTC\Delta(\In{\nLFT}-\LFTA\Delta)^{-1}\LFTB$ is the \emph{star product} \cite{Zhou1996}. 

A sufficient condition for the star product to be well-defined is as follows. 
Following \cite[Theorem 4]{Beck1999} \textcolor{red}{we call} the LFT $\LFTM$ \emph{stable},
if there exists positive definite matrices $\{P_i\}_{i=1}^d$ of size $\nLFT_i \times \nLFT_i$
such that $P=\Diag[P_1,\ldots,P_d]$ satisfy $\textcolor{red}{\LFTA^T P \LFTA} - P < 0$.
Let $\Delta$ be of the form \eqref{delt:struc:def} such that 
$\Delta$ is a linear operator from $\mathcal{L}(l_2(\mathbb{R}^{\nLFT}))$ and 
$\|\Delta\| < 1$, where
$\|\cdot\|$ is the induced operator norm.
If
$M$ is stable,  
then operator $(I-\textcolor{red}{\LFTA} \Delta)$ is invertible, the start-product $\LFTM \star \Delta$ is well-defined, 
and for any $u \in l^2(\mathbb{R}^m)$ there exists a unique solution $(z,y,u)$ of $\LFTM$ for $\Delta$, \cite[Theorem 4]{Beck1999}.
\textcolor{red}{A more general condition for well-posedness of $\LFTM \star \Delta$ can be formulated using
  \emph{structured singular values} \cite{Zhou1996}.}


 In order to avoid problems with the existence of a solution,
 we \textcolor{red}{view} LFTs as multidimensional systems \cite{Alkhoury2016,Beck1999,Ball2008}, and
 we define the concept of \emph{formal input-output map}.

\begin{Definition}\label{def:canonicalPartitioning}
  The \emph{canonical partitioning} of $\LFTM$ from \eqref{LFTform}
  is collection $\{\LFTC_i, \LFTA_{i,j},\LFTB_j\}_{i,j=1}^{d}$ of matrices 
  such that
  $\LFTA_{i,j} \in \mathbb{R}^{\nLFT_i \times \nLFT_j}$, $\LFTB_j \in \mathbb{R}^{\nLFT_j \times m}$, $\LFTC_i \in \mathbb{R}^{p \times \nLFT_i}$, 
  $i,j \in \mathbb{I}_1^d$, and 
	\begin{align*}
		\LFTA=\begin{bmatrix} 
		\LFTA_{1,1} & \LFTA_{1,2} & \cdots & \LFTA_{1,d} \\
		\LFTA_{2,1} & \LFTA_{2,2} & \cdots & \LFTA_{2,d} \\
		\vdots & \vdots   & \ddots & \vdots \\
		\LFTA_{d,1} & \LFTA_{d,2} & \cdots & \LFTA_{d,d} 
		\end{bmatrix}\!\!, & & \!\! \LFTB=\begin{bmatrix} \LFTB_1 \\ \LFTB_2 \\ \vdots \\ \LFTB_d \end{bmatrix}\!\!, & &\!\! 
    \LFTC=\begin{bmatrix} \LFTC_1^{T} \\ \LFTC_2^{T} \\ \vdots \\ \LFTC_d^{T} \end{bmatrix}^{T}\!\!\!\!,
	\end{align*}
\end{Definition}
\begin{Definition}\label{def:fiom_lfr}
  The \emph{formal input-output map} of an LFT $\LFTM$ is a
   function $Y_{\LFTM}: (\mathbb{I}_1^d)^{*} \rightarrow \mathbb{R}^{p \times m}$ such that
	$\textcolor{red}{Y_{\LFTM}(\epsilon)=\LFTD}$, and 
    for all $i_1,\ldots,i_k \in \textcolor{red}{\mathbb{I}_1^d}$ and  $k >  0$,
		\[
		\begin{split}
		& Y_{\LFTM}(i_1\cdots i_k)=\left\{\begin{array}{rl} \LFTC_{i_1}\LFTB_{i_1}  &  k = 1, \\
		\LFTC_{i_k}\LFTA_{i_k,i_{k-1}} \cdots \LFTA_{i_2,i_1}\LFTB_{i_1} & k > 1 ,
		\end{array}\right. \\
		\end{split}
		\] 
	where $\{\LFTC_i, \LFTA_{i,j},\LFTB_j\}_{i,j=1}^{d}$ is the canonical partitioning of $\LFTM$.
\end{Definition}
From \cite[Theorem 4, Proof of Lemma 11]{Beck1999}, it follows that if $\LFTM$ is stable  and for $\Delta$ of the form
\eqref{delt:struc:def} such that $\|\Delta\| < 1$
	\begin{align}
		&(\LFTMatrix \star \Delta)u = \textcolor{red}{\sum_{\nu \in (\mathbb{I}_1^d)^{*}} Y_{\LFTM}(\nu) \delta_{\nu}\I_m u}
     \!\!\! \!\! \!\! \label{eq:star1} 
	\end{align}
    \textcolor{red}{where $\delta_{\epsilon}$ is the identity operator and for $\nu=i_1\cdots i_k$, $i_1,\ldots,i_k \in \mathbb{I}_1^d$, $k> 0$,
    $\delta_{\nu}(u)$ is the composition $\delta_{i_1}\cdots \delta_{i_k}$.}
	That is, \emph{the formal input-output map $Y_{\LFTM}$ determines the star product $\LFTMatrix \star \Delta$},
  if  $\LFTM$ is stable and $\Delta$ is a contraction.
  \textcolor{red}{Note that for  uncertainity blocks arising from scheduling signals, then the resulting input-output function
   is similar to that of FALPVs: it maps inputs and scheduling signals to outputs.  
  }
  \textcolor{red}{Note than  the LFTs used in this paper need not have the same input, output and state dimensions as the FALPVs, 
hence the separate notation for these quantities.}
\textcolor{red}{Two} LFTs $\LFTM_1$ and $\LFTM_2$ are \emph{\textcolor{red}{said to be} formally input-output equivalent,}
$Y_{\LFTM_1}=Y_{\LFTM_2}$.  If $\LFTM$ is an LFT of the form \eqref{LFTform}, then $\nLFT$ the \emph{dimension} of
$\LFTM$ and we denote it by $\dim \LFTM$. We say that the LFT $\widehat{\LFTM}$ is \emph{minimal}, 
if for any LFT $\LFTM$ which is formally input-output equivalent to $\widehat{\LFTM}, \dim \widehat{\LFTM} \leq \dim \LFTM$.
%
%
\textcolor{red}{By} \cite[Theorem 20]{Beck1999} 
an LFT is minimal, if and only if its reachability and observability matrices are full row and column ranks respectively.
 \textcolor{red}{By} \cite[Theorem 8.2]{Ball2005} minimal LFTs which are formally input-output equivalent are isomorphic.
\textcolor{red}{The minimality concept above considers the number of blocks in $\Delta$ fixed: while restrictive in general, it is sufficient for our 
 purposes, as for the LFTs of interest the number of blocks in $\Delta$ is determined by the number of scheduling variables.}

\section{Transforming an FALPV model to an LFT}\label{sec:ALPVtoLFT}

\textcolor{red}{Let} $\Sigma$ be an FALPV of the form \eqref{equ:FALPVSystem}.
For any $\mathbf{p}=(\mathbf{p}_1,\ldots,\mathbf{p}_{\NP})^T \in \mathbb{P}$,
and integers  $\{\nLFT_i\}_{i=1}^d$, define the matrix 
\begin{equation}
\label{delta:const}
\Delta_{\mathbf{p}}:=\Delta_{\mathbf{p}}(\{\nLFT_i\}_{i=1}^d)\!\!=\!\! \Diag[\mathbf{p}_1 \I_{\nLFT_1}, \mathbf{p}_2 \I_{n_2}, \cdots, \mathbf{p}_{n_{d}} \I_{n_{d}}],
\end{equation}
\begin{Definition}
The function $\psi$ \emph{recognizable}, if 
there exists a minimal and stable LFT $M=(\NPSI,1,\{n_i\}_{i=1}^{\NP},F,H,G,0)$, 
called a \emph{realization of $\psi$},  such that  for any
$\mathbf{p} \in \mathbb{P}$, 
\begin{align}
\psi(\mathbf{p})=H\Delta_{\mathbf{p}}(I\!\!-\!\!F\Delta_{\mathbf{p}}))^{-1}G, ~
\label{psi:real}
\end{align}  
\textcolor{red}{where $\Delta_{\mathbf{p}}:=\Delta_{\mathbf{p}}(\{n_i\}_{i=1}^{\NP})$.}
\end{Definition}
Note that the matrix $I-F\Delta_{\mathbf{p}}$ is invertible for all $\mathbf{p} \in \mathbb{P}$.
Indeed,
$\Delta_{\mathbf{p}}$ can be identified with the linear operator on $\mathcal{L}(l_2(\mathbb{R}^{\sum_{i=1}^{\NP} n_i}))$ defined by 
$\Delta_{\mathbf{p}}z(t)=\Delta_{\mathbf{p}} z(t)$. It is easy to see that $\Delta_{\mathbf{p}}$
is of the form \eqref{delt:struc:def}  and $\|\Delta_{\mathbf{p}}\| < 1$ in the operator norm.
Hence, by \cite[Theorm 4]{Beck1999}, stability of $M$ implies that the the operator
$I-F\Delta_{\mathbf{p}}$ is invertible, 
which implies that the matrix  $I-F\Delta_{\mathbf{p}}$
is invertible for all $\mathbf{p} \in \mathbb{P}$.

\textcolor{red}{Let \emph{$\LPVS$ be the formal input-output map of the LFT $M$}. Then }
\begin{align}
\label{psi:series}
  \psi(\mathbf{p}) = \sum_{k=1}^{\infty} \sum_{i_1 \ldots i_k=1}^{\NP} 
\LPVS(i_1 \cdots i_k) 
\prod_{j=1}^{k} \mathbf{p}_{i_j}
\end{align}
and the series on the right-hand side is absolutely convergent, i.e. 
\textcolor{red}{$\psi$ is analytic and its Taylor-series coefficients are the values of $\LPVS$.}
\textcolor{red}{This explains the use of the term \emph{recognizable}: 
by \cite{Beck1999} $\LPVS$  is a recognizable formal power series \cite{Son:Real}.}
\begin{Remark}
From \cite[Theorem 1.1]{ALPAY2003225} it follows that if $\psi$ is rational, then
it has a LFT realization \textcolor{red}{$(\NPSI,1,\{n_i\}_{i=1}^{\NP},F,H,G,0)$. By choosing $\lambda=1/\|F\|_2$,
it follows that $(\NPSI,1,\{n_i\}_{i=1}^{\NP},\lambda F, \lambda H,G,0)$ is a stable LFT realization of 
of the map $\mathbf{p} \mapsto \psi(\lambda \mathbf{p})$.}
That is, rational functions, after a scaling of its arguments, are recognizable.
\end{Remark}

Next, we define the map
$\tilde{\LPVS}_{\Sigma}:(\mathbb{I}_1^{\NP})^* \to \mathbb{R}^{(\NX+\NY) \times (\NX+\NU)}$ 
\begin{equation}
\forall \nu \in (\mathbb{I}_1^{\NP})^{*}: ~ \tilde{\LPVS}_{\Sigma}(\nu)=\textcolor{red}{\sum_{l=1}^{\NPSI}} \begin{bmatrix}\LPVA_l & \LPVB_l\\\LPVC_l & \LPVD_l
\end{bmatrix} \LPVS^l(\nu). \\
\end{equation}
where $\LPVS^l(\nu)$ the $l$th component of $\LPVS(\nu)$, i.e., $\LPVS(\nu)=(\LPVS^1(\nu),\ldots,\LPVS^{\textcolor{red}{\NPSI}}(\nu))^T$. 
We need minimal and stable LFT realizations of $\tilde{\LPVS}_{\Sigma}$, existence of which is shown below.
\begin{Lemma}
  \label{lem:rational}
 \textbf{(I)}
  Let $M=(\NPSI,1,\{n_i\}_{i=1}^{\NP},F,H,G)$ be a minimal and stable LFT realization of $\psi$.
  Let  $n'_i=n_i(\NX+\NU)$, $i \in \inrange{1}{\NP}$, and define
  \begin{equation}
  \label{psi_LFT1}
  \widetilde{\mathcal{M}}(\Sigma,\psi)=(\NX+\NY,\NX+\NU, \NP, \{n'_i\}_{i=1}^{\NP},\widetilde{\mathcal{F}},\widetilde{\mathcal{G}},
  \widetilde{\mathcal{H}},0)  
  \end{equation}
   such that $\textcolor{red}{\widetilde{\mathcal{F}}}=F \otimes I_{\NX+\NU}$, $\textcolor{red}{\widetilde{\mathcal{G}}}=G \otimes I_{\NX+\NU}$,
   and 
\begin{align*}
    \widetilde{\mathcal{H}}=\sum_{l=1}^{\NPSI}H_l \otimes \begin{bmatrix}\LPVA_l & \LPVB_l\\\LPVC_l & \LPVD_l
\end{bmatrix}, ~
\end{align*}
where $H_l$ is the $l$th row of $H$, and
$\otimes$ denotes the Kronecker product.
Then $\widetilde{\mathcal{M}}(\Sigma,\psi)$ is a stable  LFT realization of 
$\tilde{\LPVS}_{\Sigma}$.

\textbf{(II)} 
Denote the  result of applying the  minimization algorithm \textcolor{red}{\cite[Theorem 7.1]{Ball2005}} to
$\widetilde{M}(\Sigma,\psi)$ by
\begin{equation}
  \label{psi_LFT}
  \mathcal{M}(\Sigma,\psi)\!\!=\!\!(\NX+\NY,\NX+\NU, \NP, \{\bar{n}_i\}_{i=1}^{\NP},\mathcal{F},\mathcal{G},\mathcal{H},0)  
  \end{equation}
 Then $\mathcal{M}(\Sigma,\psi)$ is a stable and minimal realization of $\tilde{\LPVS}_{\Sigma}$ and 
for any $\mathbf{p} \in \mathbb{P}$, with $\Delta_{\mathbf{p}}:=\Delta_{\mathbf{p}}(\{\bar{n}_i\}_{i=1}^{\NP})$,
the matrix $(\I - \mathcal{F}\Delta_{\mathbf{p}})$
is invertible and
\begin{equation}\label{eq:SumMatricesPsi}
   \sum_{l=1}^{\NPSI}\begin{bmatrix}\LPVA_l & \LPVB_l\\\LPVC_l & \LPVD_l\end{bmatrix}\
   \psi_l (\mathbf{p})= \mathcal{H} \Delta_{\mathbf{p}} (\I - \mathcal{F}\Delta_{\mathbf{p}})^{-1} \mathcal{G} 
\end{equation}
\end{Lemma}  
Consider   a minimal stable LFT realization 
$\mathcal{M}(\Sigma,\psi)$ of $\tilde{\LPVS}_{\Sigma}$ defined in \eqref{psi_LFT}, and 
the following partitioning of $\mathcal{H}$ and $\mathcal{G}$
\begin{align*}
\mathcal{H} = \begin{bmatrix}
\mathcal{H}_x^T & \mathcal{H}_y^T
\end{bmatrix}^T, \quad 
\mathcal{G} = \begin{bmatrix}
\mathcal{G}_x & \mathcal{G}_u
\end{bmatrix}
\end{align*}
such that $\mathcal{H}_x$ has $\NX$ rows, $\mathcal{H}_y$ has $\NY$ rows, $\mathcal{G}_x$ has $\NX$ columns and $\mathcal{G}_u$ has $\NU$ columns. 

We define the LFT $\mathcal{M}(\Sigma)$ \emph{associated with $\Sigma$}
 as follows.
\begin{equation}
\label{FALPV2LFT}
\mathcal{M}(\Sigma)=(\NY,\NY, \NP+1, \{\tilde{n}_i\}_{i=1}^{\NP+1}, \LFTA,\LFTB,\LFTC,\LPVD_0)
\end{equation}
where 
$\tilde{n}_1=\NX$, $\tilde{n}_i=\bar{n}_{i-1}$, $i=2,\ldots,\NP+1$, and
\begin{align*}
  \LFTA=\begin{bmatrix} \LPVA_0 & \mathcal{H}_x  \\ \mathcal{G}_x & \mathcal{F} \end{bmatrix}, ~\quad
  \LFTB=\begin{bmatrix} \LPVB_0  \\ \mathcal{G}_u \end{bmatrix}, \quad 
  \LFTC=\begin{bmatrix} \LPVC_0  &  \mathcal{H}_y \end{bmatrix}
\end{align*}  
\textcolor{red}{Note that we use a specific notation for each system class: $\LPVA,\ldots$ is used for the matrices of the FALPV $\Sigma$,
$\LFTA, \ldots $ for the matrices of a  LFT arising from  FALPV, $F,G,H$ stand for the matrices of a minimal LFT realization of $\psi$,
and $\mathcal{F},\ldots$ are the matrices of the minimal LFT realizations of $\tilde{\LPVS}_{\Sigma}$ .}

We show that $\mathcal{M}(\Sigma)$ is input-output equivalent to $\Sigma$ for
the uncertainty structures arising from scheduling signals. To this end, for any scheduling signal $p \in \mathcal{P}$, define
the \emph{uncertainty 
$\Delta(p)$ block corresponding to $p$} as
\[ \Delta(p)=\Diag[\lambda I_{\NX}, p_1 \I_{\tilde{n}_1}, \cdots, p_{\NP} \I_{\tilde{n}_{\NP}}] \]
where $\lambda$ is the backwards shift, and $p_i$ is interpreted as the element-wise product with the $i$th entry $p$,
i.e. $(\lambda z)(t)=z(t-1)$ and $(p_iz)(t)=p_i(t)z(t)$ for any sequence $z$. 
\begin{Theorem}[Main result]\label{th:FALPVtoLFT}
  Assume that $\psi$ is recognizable. 
	\begin{enumerate}
    \item
      \label{th:FALPVtoLFT:part0}
  For any solution $(x,y,u,p)$ of $\Sigma$ there exists a unique solution $(z,y,u)$ of 
  $\mathcal{M}(\Sigma)$ for $\Delta(p)$ such that $z=(x^T,\tilde{z}^T)^T$.
    \item 
                \label{th:FALPVtoLFT:part1}
		The FALPV models $\Sigma$ and $\hat{\Sigma}$ are input-output equivalent $\iff$
		$\mathcal{M}(\Sigma)$ and $\mathcal{M}(\hat{\Sigma})$ are formally input-output equivalent.
		\item
                \label{th:FALPVtoLFT:part2}
		$\Sigma$ is minimal$\iff$ $\mathcal{M}(\Sigma)$ is a  minimal.
	\end{enumerate}
	\end{Theorem}
  That is, Part \ref{th:FALPVtoLFT:part0} of  Theorem \ref{th:FALPVtoLFT} says that if $(x,y,u,p)$ is a solution of $\Sigma$
  there exists a signal $\tilde{z}$ such that $((x^T,\tilde{z}^T),y,u)$ is a solution to the
  associated LFT and it satisfies
  \begin{equation}
  \label{sol:LFT}
  \begin{split}
    &x(t+1)=\LPVA_0 x(t)+\mathcal{H}_x w(t)+\LPVB_0u(t), \\
    &\tilde{z}(t)=\mathcal{G}_x x(t)+\mathcal{F}w(t)+\mathcal{G}_u u(t),  \\
    &y(t)=\LPVC_0x(t)+\mathcal{H}_y w(t)+\LPVD u(t),~ w(t)=\Delta_{p(t)}\tilde{z}(t) \\
  \end{split}
\end{equation}
The transformation from FALPV to LFT preserves input-output behavior, \textcolor{red}{i.e.,} LFT's 
input-output behavior for uncertainty structures from scheduling variables matches that of the original FALPV. 
Since the FALPV's state is part of the LFT state, any controller stabilizing the LFT and achieving 
specific performance will do the same for the FALPV. Part \ref{th:FALPVtoLFT:part1} 
shows that input-output equivalent FALPVs yield the same LFT behavior for all uncertainty matrices. 
Part \ref{th:FALPVtoLFT:part2} ensures that minimality is preserved. In turn,
minimal, input-output equivalent LFTs are isomorphic \cite[Theorem 8.2]{Ball2005}, 
ensuring consistent controller performance regardless of the FALPV representation used for LFT generation.
\textcolor{red}{The non-zero blocks of $\Delta$ are linear functions of the scheduling variables. 
This makes the control synthesis less conservative and more 
computationally efficient.}
  \begin{Remark}[Computing $\mathcal{M}(\Sigma)$]
\label{rem:comp}
\textcolor{red}{A minimal and stable LFT realization of $\psi$ can be computed as follows: by} \eqref{psi:series}, 
a finite collection $\textcolor{red}{\mathcal{D}}=\{\LPVS(\nu)\}_{\nu \in (\mathbb{I}_1^{\NP})^{*},|\nu| \le 2n+1}$ can be recovered 
from the Taylor-series expansion of $\psi$. From \cite{Beck2001}, $\LPVS$ is a recognizable 
formal power series, and its linear representation \cite{Beck2001} can be computed from an LFT realization of $\LPVS$. 
Applying the Ho-Kalman realization algorithm \cite[Section II.B]{Son:Real} to 
\textcolor{red}{$\mathcal{D}$}
yields a minimal linear representation of $\LPVS$,\textcolor{red}{from which a stable LFT $M$ realizing $\LPVS$ can be computed using} MR factorization
\cite{Beck2001}. 
\textcolor{red}{This procedure uses linear algebra and it is computationally efficient:
its storage and time complexity is polynomial in the size of a minimal LFT realization of $\psi$ and in the size of the FALPV $\Sigma$.}
  \end{Remark}
\begin{Remark}[Computing $\mathcal{M}(\Sigma,\psi)$ and $\mathcal{M}(\Sigma)$]
  \label{rem:comp1}

\begin{color}{blue}

Once $M$ is computed, the LFT $\widetilde{\mathcal{M}}(\Sigma,\psi)$ can be
computed using \eqref{psi_LFT1}, from which $\mathcal{M}(\Sigma,\psi)$ can be
computed using \cite[Theorem 7.1]{Ball2005}, and $M(\Sigma)$ can be constructed
using \eqref{FALPV2LFT}. In some cases $\widetilde{M}(\Sigma,\psi)$ is minimal 
and it can be taken as $\mathcal{M}(\Sigma,\psi)$, as mentioned below.
\begin{Lemma}
\label{rem:comp1:lem0}
	The LFT $\widetilde{M}(\Sigma,\psi)$ satisfies the controllability rank condition of \cite[Theorem 20]{Beck1999}.
\end{Lemma}
\begin{Lemma}
\label{rem:comp1:lem}
 If  $\NPSI=1$ and
$\begin{bmatrix} \LPVA_1 & \LPVB_1\\\LPVC_1 & \LPVD_1\end{bmatrix}$ is full 

	column rank, then $\widetilde{\mathcal{M}}(\Sigma,\psi)$ is minimal. 
\end{Lemma}
In other cases,
$\mathcal{M}(\Sigma,\psi)$ can be expressed via the matrices of $M$ and $\Sigma$ 
directly, as shown below.
\begin{Lemma}
\label{rem:comp1:lem1}
If $\NPSI=1$ and 
$\begin{bmatrix} \LPVA_1 & \LPVB_1\\\LPVC_1 & \LPVD_1\end{bmatrix}=LR$, where
$L \in \mathbb{R}^{(\NY+\NX) \times m}$ is full column rank, then
	\begin{align*}
		&	\mathcal{M}(\Sigma,\psi)=(\NX+\NY,\NX+\NU, \NP, \{m n^{'}_i\}_{i=1}^{\NP},\mathcal{F},\mathcal{G},\mathcal{H},0) \\ 
	& \mathcal{F}=F \otimes I_{m}, \quad
  \mathcal{G}=G \otimes R, 
  \nonumber  \\
 &  \widetilde{\mathcal{H}}=H \otimes L
  \quad
  \forall i \in \inrange{1}{\NP}:  n'_i=n_i(\NX+\NU) \nonumber
\end{align*}
In particular, if $\LPVB_1=0,\LPVD_1=0,\LPVC_1=0$, $L=\begin{bmatrix} \LPVA \\ \mathbf{O}_{\NY \times \NX} \end{bmatrix}$, 
	$R=\begin{bmatrix} \mathbf{I}_{\NX} & \mathbf{O}_{\NX  \times \NU} \end{bmatrix}$.
\end{Lemma}

\end{color}
  \end{Remark}

\begin{color}{red}
  \begin{Example}[Polynomial dependence]
    \label{ex:polynom}
      Assume that $\Sigma$ from \eqref{equ:FALPVSystem}
      is such that $\NPSI=2$, $\LPVB_1=\LPVB_2=0$, $\LPVC_1=\LPVC_2=0$, 
        $\LPVA_1,\LPVA_2$ are non-singular, 
      $\NP=1$, $\psi_1(\mathbf{p})=\mathbf{p}$ and $\psi_2(\mathbf{p})=\mathbf{p}^2$.
    The LFT $M=(2,1,2,F,G,H,0)$ with 
    $F=\begin{bmatrix} 0 & 0 \\ 1 & 0 \end{bmatrix}$, $G=\begin{bmatrix} 1 & 0 \end{bmatrix}^T$, 
    $H=I_2$ is realization of $\psi(\mathbf{p})=(\psi_1(\mathbf{p}),\psi_2(\mathbf{p}))^T$, it satisfies the
    controllability and observability rank conditions \cite[Theorem 20]{Beck1999}.
    and it is stable ($F^TF-\I=-\I < 0$), and can be computed using Remark \ref{rem:comp}.
    
    Indeed, let us verify that $M$ is a realization of $\psi$. In this case,
    $\NP=1$, $n_2=1$, $\Delta_{\mathbf{p}}=\mathbf{p}I_2$. Hence,
    $$I-F\Delta_p=\begin{bmatrix} 1 & 0 \\ -\mathbf{p} & 1 \end{bmatrix}, $$
     and hence $(I-F\Delta_p)^{-1}=\begin{bmatrix} 1 & 0 \\ \mathbf{p} & 1 \end{bmatrix}$, and 
    therefore 
    \begin{align*}
       & H\Delta_{\mathbf{p}}(I-F\Delta_p)^{-1}G=\begin{bmatrix} \mathbf{p} & 0 \\ 0 & \mathbf{p} \end{bmatrix}\begin{bmatrix} 1 & 0 \\ \mathbf{p} & 1 \end{bmatrix}
         \begin{bmatrix} 1 \\ 0 \end{bmatrix} \\
       &= \begin{bmatrix} \mathbf{p} \\ \mathbf{p}^2 \end{bmatrix}=\psi(\mathbf{p})  
    \end{align*}
    Since the uncertainity block of $M$ contains only $1$ diagonal block, $H_1=H,F_{11}=F,G_1=G$ is the canonical partioning of $M$.
    Then the controllability and  observability matrix of $M$, as defined in \cite{Beck1999}, is given 
    \begin{align*}
     & \begin{bmatrix} G_1 &  F_{11}G_1 & \cdots  \end{bmatrix}=\begin{bmatrix} 1 & 0 & \cdots \\ 0 & 1 & \cdots \end{bmatrix} \\
      & \begin{bmatrix} H_1 \\ H_1F_{11} \\ \vdots \end{bmatrix}=
      \begin{bmatrix} 1 & 0 \\ 0 & 1 \\ 0 & 0 \\ 1  & 0 \\
      \vdots 
      & \vdots \end{bmatrix}
    \end{align*}
  and hence both are full row and column rank respectively. 
    In this case, the formal input-output map of $M$ is defined over $\mathbb{I}_1^1=\{1\}$
       $\mathcal{S}(\epsilon)=0$, and 
    \begin{align*}
       \mathcal{S}(\underbrace{1\cdots 1}_{k-times})=CF^{k-1}G=\left\{\begin{array}{rl}
           \begin{bmatrix} 1 \\ 0 \end{bmatrix} & k=1,\\
           & \\
           \begin{bmatrix} 0 \\ 1 \end{bmatrix} & k=2,\\
           & \\ 
          \begin{bmatrix} 0 \\ 0 \end{bmatrix} & k> 2
       \end{array}\right.
    \end{align*}
     In particular, the first component $\mathcal{S}^1$ of $\mathcal{S}$ is such that
     \begin{align*}
      & \mathcal{S}^1(\epsilon)=0, ~ \mathcal{S}^1(1)=1, ~\\
      &  \mathcal{S}^1(\upsilon)=0 \text{ for } \upsilon \in (\mathbb{I}_1^1)^{*}, |\upsilon|>1
     \end{align*}
     and the second component $\mathcal{S}^2$ of $\mathcal{S}$ is such that
      \begin{align*}
        & \mathcal{S}^2(\epsilon)=0, ~ \mathcal{S}^2(1)=0, ~ \mathcal{S}^2(11)=1,  \\
        & ~ \mathcal{S}^2(\upsilon)=0  \text{ for } \upsilon \in (\mathbb{I}_1^1)^{*}, |\upsilon| > 2
      \end{align*} 
     
    It then follows that
    \begin{align*}
      &  \tilde{S}_{\Sigma}(\epsilon)=\begin{bmatrix} \LPVA_1 & \LPVB_1 \\ \LPVC_1 & \LPVD_1 \end{bmatrix} \mathcal{S}^1(\epsilon)
      + \begin{bmatrix} \LPVA_2 & \LPVB_2 \\ \LPVC_2 & \LPVD_2 \end{bmatrix} \mathcal{S}^2(\epsilon)=0 \\
      & \tilde{S}_{\Sigma}(1)=\begin{bmatrix} \LPVA_1 & \LPVB_1 \\ \LPVC_1 & \LPVD_1 \end{bmatrix} \mathcal{S}^1(1)+
      \begin{bmatrix} \LPVA_2 & \LPVB_2 \\ \LPVC_2 & \LPVD_2 \end{bmatrix} \mathcal{S}^2(1)=\\
      & =\begin{bmatrix} \LPVA_1 & \LPVB_1 \\ \LPVC_1 & \LPVD_1 \end{bmatrix} = \begin{bmatrix} \LPVA_1 & 0 \\ 0 & 0 \end{bmatrix} \\
      & \tilde{S}_{\Sigma}(11)=\begin{bmatrix} \LPVA_1 & \LPVB_1 \\ \LPVC_1 & \LPVD_1 \end{bmatrix} \mathcal{S}^1(11)+
      \begin{bmatrix} \LPVA_2 & \LPVB_2 \\ \LPVC_2 & \LPVD_2 \end{bmatrix} \mathcal{S}^2(11)=\\
      & =\begin{bmatrix} \LPVA_2 & \LPVB_2 \\ \LPVC_2 & \LPVD_2 \end{bmatrix}=\begin{bmatrix} \LPVA_2 & 0 \\ 0 & 0 \end{bmatrix}  \\
      & \tilde{S}_{\Sigma}(\upsilon)=\begin{bmatrix} \LPVA_1 & \LPVB_1 \\ \LPVC_1 & \LPVD_1 \end{bmatrix} \mathcal{S}^1(\upsilon)+
      \begin{bmatrix} \LPVA_2 & \LPVB_2 \\ \LPVC_2 & \LPVD_2 \end{bmatrix} \mathcal{S}^2(\upsilon)=0 \\ 
      & \text{ for } |\upsilon|>2
    \end{align*}
    It then follows that $\widetilde{\mathcal{M}}(\Sigma,\psi)$ is such that 
    \begin{align*}
     &    \widetilde{\mathcal{F}}=\begin{bmatrix} 0 & 0 \\ 1 & 0 \end{bmatrix} \otimes \I_{\NX+\NU}
          = \begin{bmatrix} O_{\NX+\NU}  & O_{\NX+\NU} \\ \I_{\NX+\NU} & O_{\NX+\NU,\NX+\NU} \end{bmatrix}, \\
    & \widetilde{\mathcal{G}}=\begin{bmatrix} 1 \\ 0 \end{bmatrix} \otimes \I_{\NX}=
    \begin{bmatrix} \I_{\NX+\NU} \\ O_{\NX+\NU,\NX+\NU}  \end{bmatrix} \\ 
    & \widetilde{\mathcal{H}}=
       \begin{bmatrix} 1 &  0 \end{bmatrix} \otimes \begin{bmatrix} \LPVA_1 & \LPVB_1 \\ \LPVC_1 & \LPVD_1 \end{bmatrix} 
         + \begin{bmatrix} 0 & 1 \end{bmatrix} \otimes \begin{bmatrix} \LPVA_2 & \LPVB_2 \\ \LPVC_2 & \LPVD_2 \end{bmatrix} \\
       & =
       \begin{bmatrix} 1 &  0 \end{bmatrix} \otimes \begin{bmatrix} \LPVA_1 & O_{\NX,\NU} \\ O_{\NY,\NX} & O_{\NY,\NU} \end{bmatrix} \\ 
        & + \begin{bmatrix} 0 & 1 \end{bmatrix} \otimes \begin{bmatrix} \LPVA_2 & O_{\NX,\NU} \\ O_{\NY,\NX} & O_{\NY,\NU} \end{bmatrix} \\
            & = \begin{bmatrix} \LPVA_1 & O_{\NX,\NU} & O_{\NX,\NX} & O_{\NX,\NU} \\ 
            O_{\NY,\NX} & O_{\NY,\NU} & O_{\NY,\NX} & O_{\NY,\NU} \end{bmatrix} + \\ 
            & \begin{bmatrix} O_{\NX,\NX}  & O_{\NX,\NU} & \LPVA_2 & O_{\NX,\NU}  \\ 
         O_{\NY,\NX} & O_{\NY,\NU} & O_{\NY,\NX} & O_{\NY,\NU}  
            \end{bmatrix}  \\
        & =\begin{bmatrix} \LPVA_1 & O_{\NX,\NU} & \LPVA_2 &  O_{\NX,\NU} \\ 
           O_{\NY,\NX} & O_{\NY,\NU} & O_{\NY,\NX} & O_{\NY,\NU}  \end{bmatrix}
  \end{align*}
  where $O_{k,l}$ is the $k \times l$ zero matrix.
  For the LFT $\widetilde{M}(\Sigma,\psi)$, we have that
  its canonical partitioning is given by $\widetilde{F}_{11}=\widetilde{F}$,
  $\widetilde{H}_1=\widetilde{H}$, and $\widetilde{G}_1=\widetilde{G}$, as
  the uncertainty block $\Delta_{\mathbf{p}}$ has one diagonal block. 
  Notice that $\widetilde{F}_{11}^2=\widetilde{F}^2=0$ and 
  \begin{align*}
    \widetilde{F}_{11}\widetilde{G}_{1}=\widetilde{F}\widetilde{G}=\begin{bmatrix} 
    O_{\NX+\NU,\NX+\NU} \\ \I_{\NX+\NU,\NX+\NU} \end{bmatrix}
  \end{align*}
  In this case, the controllability and observability matrices of $\widetilde{M}(\Sigma,\psi)$ are given by
  \begin{align*}
    & R=\begin{bmatrix} \widetilde{G}_1 & \widetilde{F}_{11}\widetilde{G}_1 & \cdots \end{bmatrix} \\
      & =\begin{bmatrix} \I_{\NX+\NU} &  O_{\NX+\NU,\NX+\NU} & \cdots  \\ O_{\NX+\NU,\NX+\NU} & \I_{\NX+\NU} & \cdots \end{bmatrix}=\\
          & = \begin{bmatrix} \I_{2(\NX+\NU)} & \cdots \end{bmatrix} \\
      & O=\begin{bmatrix} \widetilde{H}_1 \\ \widetilde{H}_1\widetilde{F}_{11} \end{bmatrix}=
        \begin{bmatrix} \widetilde{H} \\ \widetilde{H}\widetilde{F} \end{bmatrix}= \\
      & =\begin{bmatrix} \LPVA_1 & O_{\NX,\NU} & O_{\NX,\NX} & O_{\NX,\NY} \\
        O_{\NY,\NX} & O_{\NX,\NU} & O_{\NU,\NX} & O_{\NX,\NU} \\ 
          \LPVA_2 & O_{\NX,\NU} & O_{\NX,\NU} & O_{\NX,\NU} \\ 
               O_{\NY,\NX} & O_{\NX,\NU} & O_{\NX,\NX} & O_{\NY,\NU}  \\
               \vdots & \vdots & \vdots & \vdots \\
            \end{bmatrix}  
  \end{align*}  
   It then follows that $R$ is full row rank. Since $\LPVA_1,\LPVA_2$ are non-singular, it follows that
   \[ \ker O= \mathrm{Im}  \begin{bmatrix} O_{\NX,\NX} \\ \I_{\NU} \\ O_{\NX,\NX} \\ \I_{\NU} \end{bmatrix} \]
   Consider the basis transformation
   \[ T=\begin{bmatrix} \I_{\NX} & O_{\NX,\NU} & O_{\NX,\NX} & O_{\NX,\NU} \\
                         O_{\NX,\NX}       & O_{\NX,\NU}  & \I_{\NX} & O_{\NX,\NU} \\
                        O_{\NU,\NX}        & \I_{\NU} & O_{\NU,\NX} & O_{\NU,\NU} \\
                        O_{\NU,\NX}       & O_{\NU,\NU}      &  O_{\NU,\NX} & \I_{\NU}
                        \end{bmatrix}
    \]
    By direct computation it follows that  $T*T=I_{2(\NX+\NU)}$ and hence $T=T^{-1}$.
    Then the transformed LFT with $T\widetilde{F}T^{-1}$, $T\widetilde{G}$, $\widetilde{H}T^{-1}$
    is in Kalman-decomposition form \cite[Theorem 7.1]{Ball2005} and is of the form
    \begin{align*}
          & T\widetilde{F}T^{-1}=              
          \begin{bmatrix} O_{\NX,\NX} O_{\NX,\NX} & O_{\NX,\NU}  & O_{\NX,\NU} \\ 
                        \I_{\NX,\NX} & O_{\NX,\NX} & O_{\NX,\NU} & O_{\NX,\NU} \\
                        O_{\NU,\NX} & O_{\NU,\NX} & O_{\NU,\NU} & O_{\NU,\NU} \\
                        O_{\NU,\NX} & O_{\NU,\NX} & \I_{\NU} & O_{\NU,\NU} \\
                      \end{bmatrix}, \\
          & T\widetilde{G}=\begin{bmatrix} \I_{\NX}  & 0_{\NX,\NU} 
                                   \\ O_{\NX,\NX} & O_{\NX,\NU} \\ 
                                   O_{\NU,\NX} & \I_{\NU,\NU} \\
                                    O_{\NU,\NX} & O_{\NU,\NU}
                                  \end{bmatrix}, \\            
          & \widetilde{H}T^{-1}=\begin{bmatrix} \LPVA_1 & \LPVA_2 & O_{\NX,\NU} & O_{\NX,\NU} \\ 
                        O_{\NY,\NX} & O_{\NY,\NX} & O_{\NY,\NU} & O_{\NY,\NU} 
                      \end{bmatrix}
    \end{align*}        
     The the matrices $\mathcal{F},\mathcal{G},\mathcal{H}$, given by
     \begin{align*}
        & \mathcal{F}=\begin{bmatrix} O_{\NX,\NX} & O_{\NX,\NX} \\
                                    \I_{\NX} & O_{\NX,\NX} 
        \end{bmatrix},  \\
        & \mathcal{G}=\begin{bmatrix} \I_{\NX}  &  O_{\NX,\NU} \\ 
                      O_{\NX,\NX} & O_{\NX,\NU} \end{bmatrix}, \\
        & \mathcal{H}=
        \begin{bmatrix} \LPVA_1 & \LPVA_2 \\
                                    O_{\NY,\NX} & O_{\NY,\NX} 
        \end{bmatrix} 
    \end{align*}
    correspod to the minimal subsystem of the Kalman decomposition of 
    $\widetilde{M}(\Sigma,\psi)$. That is,
    we can take $\mathcal{M}(\Sigma,\psi)$ as in \eqref{psi_LFT}, 
    with the above choide of matrices $\mathcal{F},\mathcal{G},\mathcal{H}$. 
   
    In fact, a direct computation reveals that
    \begin{align*}
      &\sum_{l=1}^{2}\begin{bmatrix}\LPVA_l & \LPVB_l\\\LPVC_l & \LPVD_l\end{bmatrix}\
   \psi_l (\mathbf{p})
   = \begin{bmatrix} \LPVA_1 \mathbf{p}+\LPVA_2 \mathbf{p}^2 & O_{\NX,\NU}\\
    O_{\NY,\NX} & O_{\NY,\NU} \end{bmatrix} \\
    & = \mathcal{H} \Delta_{\mathbf{p}}(I-\mathcal{F}\Delta_{\mathbf{p}})^{-1}\mathcal{G}= \\
    & \begin{bmatrix} \LPVA_1\mathbf{p} & \LPVA_2\mathbf{p} \\ O_{\NY,\NX} & O_{\NY,\NU} \end{bmatrix}
      (\I_{2\NX,}-\begin{bmatrix} O_{\NX,\NX} & O_{\NX,\NU} \\ \mathbf{p}\I_{\NX,\NX} & O_{\NX,\NU} \end{bmatrix})^{-1}
       \\
      & \times \begin{bmatrix} \I_{\NX} & O_{\NX,\NU}\\ O_{\NX,\NX} & O_{\NX,\NU} \end{bmatrix} \\
    & = 
    \begin{bmatrix} \LPVA_1\mathbf{p} & \LPVA_2\mathbf{p} \\ O_{\NY,\NX} & O_{\NY,\NU} \end{bmatrix}
      \begin{bmatrix} \I_{\NX} & O_{\NX,\NU}\\ -\mathbf{p}\I_{\NX,\NX} & I_{\NX} \end{bmatrix}^{-1}
      \\
     &  \times \begin{bmatrix} \I_{\NX} & O_{\NX,\NU}\\ O_{\NX,\NX} & O_{\NX,\NU} \end{bmatrix} \\
     &  
    =\begin{bmatrix} \LPVA_1\mathbf{p} & \LPVA_2\mathbf{p} \\ O_{\NY,\NX} & O_{\NY,\NU} \end{bmatrix}
      \begin{bmatrix} \I_{\NX} & O_{\NX,\NU}\\ \mathbf{p}\I_{\NX,\NX} & I_{\NX} \end{bmatrix} \\
      & \times \begin{bmatrix} \I_{\NX} & O_{\NX,\NU}\\ O_{\NX,\NX} & O_{\NX,\NU} \end{bmatrix} \\
    \end{align*}
    i.e., \eqref{eq:SumMatricesPsi} holds.
    Moreover,  the controllability and observability matrices of $\mathcal{M}(\Sigma,\psi)$ are given by
    \begin{align*}
       & R=\begin{bmatrix} \mathcal{G} & \mathcal{F}\mathcal{G} & \cdots \end{bmatrix}=
           \begin{bmatrix} \I_{\NX} & O_{\NX,\NX} \\ O_{\NX,\NX} & \I_{\NX} \end{bmatrix} \\
       & O=\begin{bmatrix} \mathcal{H} \\ \mathcal{H}\mathcal{F} \end{bmatrix}=
          \begin{bmatrix} \LPVA_1 & \LPVA_2 \\ 
                        O_{\NY,\NX} & O_{\NY,\NX} \\
                        \LPVA_2 & O_{\NX,\NX} \\
                        O_{\NY,\NX} & O_{\NY,\NX} \\
                        \vdots  & \vdots 
        \end{bmatrix} 
      \end{align*}
      and they are full row and column rank respectively, as $\LPVA_1,\LPVA_2$ are non-singular. Hence,
      by \cite{Ball2005}, the LFT $\mathcal{M}(\Sigma,\psi)$ is minimal.
    Then 
    \begin{align*}
      & \mathcal{G}_x=\begin{bmatrix} I_{\NX} \\ O_{\NX,\NX}  \end{bmatrix}, ~
      \mathcal{G}_u=O_{2\NX,\NU} \\
      & \mathcal{H}_x=\begin{bmatrix} \LPVA_1 & \LPVA_2 \end{bmatrix}, ~
      \mathcal{H}_y=O_{\NY,2\NX}
    \end{align*}
    Then $\mathcal{M}(\Sigma)$ is of the form \eqref{FALPV2LFT}, with  $\LPVD_0=O_{\NY,\NU}$,
    \begin{align*}
    & \LFTA=\begin{bmatrix} \LPVA_0 & \LPVA_1 & \LPVA_2  \\ 
                          \I_{\NX} & O_{\NX,\NX}     & O_{\NX,\NX}   &  \\
                          O_{\NX,\NX}      & I_{\NX}    & O_{\NX,\NX}   & 
          \end{bmatrix}, ~  \LFTB=\begin{bmatrix} \LPVB_0 \\ O_{\NX,\NU} \\ O_{\NX,\NU}\end{bmatrix}, \\
          & \LFTC=\begin{bmatrix} \LPVC_0 & O_{\NY,\NX} &  O_{\NY,\NX} \end{bmatrix}, 
    \end{align*}
and the uncertainity block is $\Delta(p)=p(t)I_{2\NX}$. 
That is, there is one diagonal block which dependens linearly on $p$.
A direct substitution shows that  any solution $(x,y,u,p)$ of $\Sigma$ satisfies \eqref{sol:LFT} 
with $\tilde{z}=(x^T,x^Tp)^T$.

    \end{Example}

\begin{Example}[Rational dependence]\label{ex:FALPVtoDiagDzw}
Consider the FALPV $\Sigma$ of the form \eqref{equ:FALPVSystem}, such that 
$\NP=1$, $\NPSI=1$, $\LPVA_0=0$, $\LPVB_1=0$, $\LPVC_1=0$, $\LPVA_1$ non-singular, and 
$\psi(\mathbf{p})=\frac{\mathbf{p}}{1-0.5\mathbf{p}}$.
Then  $M=(1,1,1,F,G,H)$ with $F=0.5$, $H=G=1$ is stable and minimal LFT realizing $\psi$, and it 
can be computed from $\psi$ by using Remark \ref{rem:comp}. 
Indeed, in thise case $\Delta_{\mathbf{p}}=\mathbf{p}$ and 
\[ 
        H\Delta_{\mathbf{p}}(I-F\Delta_{\mathbf{p}})^{-1}G=\mathbf{p}(1-0.5\mathbf{p})^{-1}=\psi(\mathbf{p}).
\]
The dimension of $M$ is $1$, i.e., it is minimal.

The formal input-output map $\mathcal{S}$ of $M$ is defined over the set of sequences
$(\mathbb{I}_1^1)^{*}$,  $\mathbb{I}_1^1=\{1\}$, i.e., all sequences of $1$s. 
\begin{align*}
    \mathcal{S}(\epsilon)=0, ~ \mathcal{S}(\underbrace{1\cdots 1}_{k \mbox{ times }})=0.5^{k-1}
\end{align*}
The map $\tilde{S}_{\Sigma}$ is then defined as
\begin{align*}
  \tilde{S}_{\Sigma}(\upsilon)=\begin{bmatrix} \LPVA_1 & O_{\NX,\NU} \\ O_{\NY,\NX} & O_{\NY,\NU} \end{bmatrix} 0.5^{|\upsilon|-1}, 
  |\upsilon|>1, \text{ and } \tilde{S}_{\Sigma}(\epsilon)=0
\end{align*}
where $O_{k,l}$ is the $k \times l$ zero matrix.
Then an LFT $\widetilde{\mathcal{M}}(\Sigma,\psi)$
realizing $\tilde{S}_{\Sigma}$ is given by
\begin{align*}
&   \widetilde{\mathcal{F}} = F \otimes \I_{\NX+\NU}=0.5\I_{\NX+\NU} \\
&    \widetilde{\mathcal{G}}=G \otimes \I_{\NX+\NU}=\I_{\NX+\NU}, \\
 &  \widetilde{\mathcal{H}}=H \otimes \begin{bmatrix} \LPVA_1 & O_{\NX,\NU}\\ O_{\NY,\NX} & O_{\NY,\NU} \end{bmatrix}=
   \begin{bmatrix} \LPVA_1 & O_{\NX,\NU} \\ O_{\NY,\NX} & O_{\NY,\NU} \end{bmatrix}.
\end{align*}
It is easy to see that the LFT $\widetilde{\mathcal{M}}(\Sigma,\psi)$ is stable and minimal.
Indeed, $\widetilde{\mathcal{F}}^T\widetilde{\mathcal{F}}=0.5^2\I_{\NX+\NU} < \I_{\NX+\NU}$
and the controllability and observability matrices are given by
\begin{align*}
    & R=\begin{bmatrix} \widetilde{\mathcal{G}} & \cdots \end{bmatrix}=\begin{bmatrix} \I_{\NX+\NU} & \cdots \end{bmatrix}, \\
    & O=\begin{bmatrix} \widetilde{\mathcal{H}} \\ \vdots \\ \widetilde{\mathcal{H}}\widetilde{\mathcal{F}}^k \\ \vdots  \end{bmatrix}=
    \begin{bmatrix} \LPVA_1 & O_{\NX,\NU} \\ O_{\NY,\NX} & O_{\NY,\NU}  \\ \vdots & \vdots \\ \LPVA_1 0.5^k & O_{\NY,\NU} \\ \vdots &  \vdots  \end{bmatrix} 
\end{align*}
That is the controllability matrix is full row rank, and the last $\NU$ colums of the
 observability matrix is zero, and the first $\NX$ columns are linearly independent, as
 $\LPVA_1$ is non-singular.
 Then the LFT $\widetilde{\mathcal{M}}(\Sigma,\psi)$ is already in Kalman-decomposition form \cite[Theorem 7.1]{Ball2005}.
 and the minimal subsystem is given by first $\NX$ states of the LFT, i.e., if we define
 \begin{align*}
  \mathcal{F}=0.5\I_{\NX}, ~ \mathcal{G}=\begin{bmatrix} \I_{\NX} \\ O_{\NX,\NU} \end{bmatrix}, ~
   \mathcal{H}=\begin{bmatrix} \LPVA_1 \\ O_{\NY,\NX} \end{bmatrix} 
 \end{align*}
 In fact, by a direct computation, we have that
 for any $\upsilon \in (\mathbb{I}_1^1)^{*}$,
 $\tilde{S}_{\Sigma}(\upsilon)=\mathcal{H} \mathcal{G}=O_{\NX+NY,\NX+\NU}$,  and 
 if $|\upsilon|>1$, then
 \begin{align*}
  & \tilde{S}_{\Sigma}(\upsilon)=\begin{bmatrix} \LPVA_1 & O_{\NX,\NU} \\ O_{\NY,\NX} & O_{\NY,\NU} 
   \end{bmatrix} 0.5^{|\upsilon|-1} \\
  & = \begin{bmatrix} \LPVA_1 & O_{\NX,\NU} \\ O_{\NY,\NX} & O_{\NY,\NU} \end{bmatrix}  
  (0.5 \I_{\NX})^{|\upsilon|-1}\begin{bmatrix} \I_{\NX} \\ O_{\NX,\NU} \end{bmatrix}=
  \mathcal{H} \mathcal{F}^{|\upsilon|-1} \mathcal{G}. 
 \end{align*}
 i.e., $\mathcal{M}(\Sigma,\psi)$ is indeed a realization of $\tilde{S}_{\Sigma}$.
 Moreover, clearly the controllability and observability matrices of $\mathcal{M}(\Sigma,\psi)$ are of the form
 \begin{align*}
    & R=\begin{bmatrix} \mathcal{G} & \cdots \end{bmatrix}=\begin{bmatrix} \I_{\NX} & O_{\NX,\NU} & \cdots \end{bmatrix} \\
    & O=\begin{bmatrix} \mathcal{H} \\ \vdots \end{bmatrix}=\begin{bmatrix} \LPVA_1 \\ \vdots \end{bmatrix}
 \end{align*}
and they are   full row and column rank respectively, as $\LPVA_1$ is non-singular.
That is, $\mathcal{M}(\Sigma,\psi)$ is minimal.
It then follows that
\begin{align*}
    & \mathcal{G}_x=I_{\NX} ~
    \mathcal{G}_u=O_{\NX,\NU} \\
    & \mathcal{H}_x=\LPVA_1, \mathcal{H}_y=O_{\NY,\NX}
\end{align*}
 Then  $\mathcal{M}(\Sigma)$ is given by
 \begin{align*}
   & \LFTA=\begin{bmatrix} O_{\NX,\NX} &  \LPVA_1 \\ \I_{\NX} & 0.5I_{\NX} \end{bmatrix} \\
   & \LFTB=\begin{bmatrix} \LPVB_0 \\ O_{\NX,\NU} \end{bmatrix} \\
   & \LFTC=\begin{bmatrix} \LPVC_0 &  O_{\NY,\NX} \end{bmatrix}, ~ \LPVD_0=O_{\NY,\NU}
 \end{align*}
 and the uncertainity block is $\Delta(p)=p(t)I_{\NX}$.
Then any solution $(x,y,u,p)$ of $\Sigma$ satisfies \eqref{sol:LFT} with $\tilde{z}= \frac{x}{1-0.5p}$.
\end{Example}
\end{color}
\begin{color}{blue}
\begin{Remark}[Comparison with existing work]
  \label{rem:related}
FALPVs with recognizable non-linear functions $\psi$ are special cases of LPV 
systems with rational dependence on scheduling variables. Existing methods, 
e.g., \cite{scherer2000linear,Hecker01112006,Zhou1996}, construct LFTs by 
processing rational entries of LPV matrices symbolically and minimizing the 
resulting LFTs. However, these methods may yield LFTs of varying sizes and 
formal input-output behavior for the same FALPV system, as noted in 
\cite[Section 4]{Hecker01112006}. 
\begin{Example}[Non-equivalent LFTs of different dimensions]
\label{example2}
For example, consider a FALPV of the form \eqref{equ:FALPVSystem}, such that 
$\NP=2$, $\NPSI=1$, $\LPVA_0=0$, $\LPVB_1=0$, $\LPVC_1=0$ and 
\begin{align*}
 	\LPVA_0=\begin{bmatrix} 1 & 0 \\ 0 & 0 \end{bmatrix}, \quad \LPVA_1\begin{bmatrix} 1 & 0 \\ 2 & 1 \end{bmatrix} \\
		\LPVB_0=\begin{bmatrix} 0 & 1 \end{bmatrix}^T, \quad \LPVC_0=\begin{bmatrix} 1 & 0 \end{bmatrix}, \quad \LPVD_0=0	
\end{align*}
Moreover, 
\[ \psi(\mathbf{p})=3\mathbf{p}_1+(3-4\mathbf{p}_2+(-4+\mathbf{p}_2+\mathbf{p}_1)\mathbf{p}_1)\mathbf{p}_1. \]
Applying standard methods of the Matlab GSS package \url{https://w3.onera.fr/smac/gss} yields an LFT of dimension $10$.
If we replace the FALPV by the isomorphic FALPV with matrices by choosing $T=\begin{bmatrix} 1 & 2 \\ 3 & 4 \end{bmatrix}$
and replacing $\LPVA_i,\LPVB_i, \LPVC_i$ by $T \LPVA_i T^{-1}, T\LPVB_i,\LPVC_iT^{-1}$ for all $i=0,1$ and we use the following
alternative equivalent representation of $\psi$
\[ \psi(\mathbf{p})= \mathbf{p}_1^3+\mathbf{p}_2*\mathbf{p}_1^2-4*\mathbf{p1}^2-4*\mathbf{p}_1*\mathbf{p}_2+3*\mathbf{p}_1+3*\mathbf{p}_2 \]
and we apply again  Matlab GSS package we obtain an LFT of dimension $12$. Moreover, by computing the difference between the LFTs obtained for these two systems
using the same package, we can see that these two LFTs are not formally input-output equivalent. 
\end{Example}
This 
discrepancy arises from differences in symbolic processing, e.g., treating 
$\mathbf{p}_1\mathbf{p}_2$ and $\mathbf{p}_2\mathbf{p}_1$ differently. 
\begin{Example}
\label{example3}
	Let us apply the GLSS toolbox to the first FALPV from Example \ref{example2} for $\psi(\mathbf{p})=\mathbf{p}_1\mathbf{p}_2$ and then to the second one, but
with $\psi(\mathbf{p})$ written as $\psi(\mathbf{p})=\mathbf{p}_2\mathbf{p}_1$. The resulting two LFTs will have the same dimension $4$, but their difference
is not zero, i.e., they are not formally input-output equivalent. This can also be seen by analyzing the two LFTs. The first LFT is of the form
\eqref{FALPV2LFT} with $\tilde{n}_1=2, \tilde{n}_2=\tilde{n}_3=1$,
\begin{align*}
	\LFTA=\begin{bmatrix} \LPVA_0 & \LPVA_1 \\ F_{11} & F_{12} \end{bmatrix}, \quad \LFTB=\begin{bmatrix} \LPVB_0 \\ \mathbf{O}_{2,1} \end{bmatrix}, 
		\LFTC=\begin{bmatrix} \LPVC_0 & \mathbf{O}_{1,2} \end{bmatrix}
\end{align*}
where 
\[
	F_{11}=\begin{bmatrix} 0 & 0 \\ 0 & 1 \end{bmatrix}, \quad F_{12}=\begin{bmatrix} 0 & 1 \\ 0 & 0 \end{bmatrix}
\]
while the second LFT is of the form \eqref{FALPV2LFT} with $\tilde{n}_1=2, \tilde{n}_2=\tilde{n}_3=1$
\begin{align*}
	\LFTA=\begin{bmatrix} T\LPVA_0T^{-1} & \hat{F}_{0,1} \\ \hat{F}_{11} & \hat{F}_{12} \end{bmatrix}, \quad \LFTB=\begin{bmatrix} T\LPVB_0 \\ \mathbf{O}_{2,1} \end{bmatrix}, 
		\LFTC=\begin{bmatrix} \LPVC_0 T^{-1} & \mathbf{O}_{1,2} \end{bmatrix}
\end{align*}
where 
\begin{align*}
	& \hat{F}_{0,1}=\begin{bmatrix}  0  &  -6.3246  \\
	0  & -15.8114 \end{bmatrix} \\
	& \hat{F}_{11}=\begin{bmatrix} -0.9487  &  0.3162 \\
		0  &       0 \end{bmatrix}, \quad 
		\hat{F}_{12}=\begin{bmatrix}     0    &     0  \\
		1 &         0   \end{bmatrix}
\end{align*}
We can see that the two LFTs are not isomorphic, and in fact not even formally input-output equivalent.
\end{Example}

Additionally, minimization steps obscure the relationship between FALPV 
matrices and resulting LFTs, complicating differentiation between the effects 
of FALPV matrices and $\psi$.

Preservation of minimality, input-output equivalence, and isomorphism during 
FALPV-to-LFT transformation has not been formally addressed in the literature.

This paper refines existing methods by focusing on constructing a minimal LFT 
realization of the non-linear function $\psi$ only. This realization is then 
used to transform FALPVs into LFTs that inherit the minimality, input-output 
equivalence, and isomorphism properties of the FALPVs.

The resulting LFT explicitly depends on the matrices of the FALPV 
$\Sigma$ under mild assumptions, as mentioned in Remark \ref{rem:comp1}, see  
for details. This explicit dependence facilitates 

To this end, recall from \cite{Alkhoury2017a} the notions of parametrization, 
structural identifiability and minimality for FALPVs. Denote by 
$\mathcal{LPV}(\NP,\NX,\NU,\NY)$ the set of all FALPV models of the form 
\eqref{equ:FALPVSystem}. A FALPV parametrization is a function 
$\mathbf{\Sigma}:\Theta \rightarrow \mathcal{LPV}(\NP,\NX,\NU,\NY)$.

A FALPV parametrization $\mathbf{L}$ is \emph{structurally identifiable}, if 
for any two distinct parameter values $\theta_1, \theta_2 \in \Theta$, 
$\theta_1 \ne \theta_2$, the input-output maps of $\mathbf{L}(\theta_1)$ and 
$\mathbf{L}(\theta_2)$ are not equal. A FALPV parametrization $\mathbf{L}$ is 
\emph{structurally minimal}, if for any parameter value $\theta \in \Theta$, 
$\mathbf{L}(\theta)$ is minimal. In \cite{Alkhoury2017a} it is shown that 
structurally minimal FALPV parametrizations are structurally identifiable, if 
and only if they contain no two distinct isomorphic FALPV systems.

We say that an LFT parametrization $\mathbf{M}$ \emph{originates from a FALPV 
parametrization $\mathbf{L}$}, if for every $\theta \in \Theta$, 
$\mathbf{M}(\theta)$ is LFT $\mathcal{M}(\mathbf{L}(\theta))$ which is 
calculated from the FALPV $\mathbf{L}(\theta)$ as in \eqref{FALPV2LFT}. 

We say that the LFT parametrization $\mathbf{M}$ is \emph{structurally 
identifiable}, if for any two distinct parameter values $\theta_1, \theta_2 \in 
\Theta$, $\theta_1 \ne \theta_2$, $\mathbf{M}(\theta_1)$ and 
$\mathbf{M}(\theta_2)$ are not formally input-output equivalent. Finally, 
$\mathbf{M}$ is called \emph{structurally minimal}, if for every parameter 
value $\theta \in \Theta$, $\mathbf{M}(\theta)$ is a minimal LFT.

Intuitively, for structurally identifiable LFT parametrizations, there is a 
one-to-one correspondence between parameters and the input-output behavior of 
the corresponding LFT for any uncertainty structure which respects the fixed 
block-structure, in particular, for uncertainty blocks arising from scheduling 
signals. Consider a FALPV parametrization $\mathbf{L}$ and an LFT 
parametrization $\mathbf{M}$ arising from $\mathbf{L}$. Then \emph{$\mathbf{L}$ 
is structurally minimal respectively structurally identifiable $\iff$  
$\mathbf{M}$ is structurally minimal, respectively structurally identifiable. }

\end{Remark}
\end{color}

\section{Proofs}
\label{proof}

For the proofs, we introduce the following notation.
\begin{Notation}
\label{lfr_proof:not1}
Consider the canonical partitioning from Definition \ref{def:canonicalPartitioning} of
an LFT $\mathcal{M}$ of the form \eqref{LFTform}.
If $\nu=i_1i_2\cdots i_k \in (\mathbb{I}_1^d)^{*}$, $i_1,\ldots, i_k \in \mathbb{I}_1^d$, $k \ge 1 $, then 
\( [\LFTA]_{\nu} =
    \LFTA_{i_k,i_{k-1}}\LFTA_{i_{k-1},i_{k-2}} \cdots \LFTA_{i_2,i_1} \) for $k > 1$, and  
    $[\LFTA]_{i_1}$ is the identity matrix of size $d_{i_1} \times d_{i_1}$ if $k=1$.
For any  $i,j \in \mathbb{I}_1^d$, for any $\nu \in (\mathbb{I}_1^d)^{*}$,  let
 $[\LFTA \# \LFTB]_{i \nu} = [\LFTA]_{\i \nu} \LFTB_i$, 
$[\LFTC \# \LFTA]_{\nu i} =\LFTC_i [\LFTA]_{\nu i}$, and
 $[\LFTC \# \LFTA \# \LFTB]_{i \nu j }= \LFTC_i [\LFTA]_{i \nu j}\LFTB_j$
\end{Notation}
\begin{proof}[Proof of Lemma\ref{lem:rational}]
  \textbf{(I)}
  Stability of $\widetilde{\mathcal{M}}(\Sigma,\psi)$ follows from that of $M$: $Y=\Diag[P_1,\ldots,P_{\NP}] >0$
  a structured matrix such that $F^TYF - Y < 0$, then for $\bar{Y}=Y \otimes \I_{\NX+\NU}$,
  $\widetilde{\mathcal{F}}^T\bar{Y}\widetilde{\mathcal{F}} - \bar{Y}=(F^TYF-Y) \otimes \I_{\NX+\NU}$, 
  and by
  \cite[page 708]{golub2013matrix},  $(F^TYF-Y) \otimes \I_{\NX+\NU} < 0$.

  Concerning $\widetilde{\mathcal{M}}(\Sigma,\psi)$ being a realization of $\tilde{\LPVS}_{\Sigma}$,
   notice that if $\{(H_i,F_{i,j},G_j)\}_{i,j=1}^{\NP}$ 
  and $\{(\widetilde{\mathcal{H}}_i,\widetilde{\mathcal{F}}_{i,j},\widetilde{\mathcal{G}}_j)\}_{i,j=1}^{\NP}$ is
   the canonical partioning of $M$ and $\widetilde{\mathcal{M}}(\Sigma,\psi)$ respectively, then 
   $\widetilde{\mathcal{F}}_{i,j}=F_{i,j} \otimes \I_{\NX+\NU}$, 
   $\widetilde{\mathcal{G}}_{j}=G_{j} \otimes \I_{\NX+\NU}$, 
   and $\widetilde{\mathcal{H}}_i= \sum_{l=1}^{\NPSI} (H_i)_l \otimes \begin{bmatrix}\LPVA_l & \LPVB_l\\\LPVC_l & \LPVD_l
  \end{bmatrix}$
  where $(H_i)_l$ denotes the $l$th row of $H_i$. 
  The rest of the proof follows by repeated application of  $(A \otimes B)(C \otimes D)=AC \otimes BD$.

  Since $\widetilde{\mathcal{M}}(\Sigma,\psi)$ is stable, by \cite[Theorem 4]{Beck1999} it follows that 
  $\I-\widetilde{F}\tilde{\Delta}_{\mathbf{p}}$ is invertible for $\widetilde{\Delta}_{\mathbf{p}}=\Delta_{\mathbf{p}}(\{n'_i\}_{i=1}^{\NP})$.
  Notice that $\psi_l(\mathbf{p})=H_l\bar{\Delta}_{\mathbf{p}}(\I-F\bar{\Delta}_{\mathbf{p}})^{-1}G$,
  where $\bar{\Delta}_{\mathbf{p}}$ is of the form \eqref{delta:const} for $d=\NP$ and $\nLFT_i=n_i$, $i \in \mathbb{I}_1^{\NP}$.
  Then $\widetilde{\Delta}_{\mathbf{p}}=\bar{\Delta}_{\mathbf{p}} \otimes \I_{\NX+\NU}$, and 
  by using that the Kronecker product being distributive w.r.t.
  matrix multiplication \cite[Section 12.3.1]{golub2013matrix}, 
  it follows that $\sum_{l=1}^{\NPSI}\begin{bmatrix}\LPVA_l & \LPVB_l\\\LPVC_l & \LPVD_l\end{bmatrix}\
  \psi_l (\mathbf{p})= \widetilde{\mathcal{H}} \widetilde{\Delta}_{\mathbf{p}} (\I - \widetilde{\mathcal{F}}\widetilde{\Delta}_{\mathbf{p}})^{-1} \widetilde{\mathcal{G}}$.
  

  \textbf{(II)}  $\mathcal{M}(\Sigma,\psi)$ is a minimal stable realization of $\tilde{\LPVS}_{\Sigma}$ by
  \cite[proof Lemma 11]{Beck1999}. \textcolor{red}{By} \cite[Theorem 9]{Beck1999}, 
  $\mathcal{M}(\Sigma,\psi) \star \Delta_{\mathbf{p}} = \widetilde{\mathcal{M}}(\Sigma,\psi) \star \widetilde{\Delta}_{\mathbf{p}}$,
  where $\Delta_{\mathbf{p}}$, $\widetilde{\Delta}_{\mathbf{p}}$ are interpreted as operators on sequences acting element-wise:
  $(\Delta_{\mathbf{p}}z)(t)=\Delta_{\mathbf{p}}z(t)$, 
  $(\widetilde{\Delta}_{\mathbf{p}}z)(t)=\widetilde{\Delta}_{\mathbf{p}}z(t)$.
  Then  \( \widetilde{\mathcal{H}} \widetilde{\Delta}_{\mathbf{p}} (\I - \widetilde{\mathcal{F}}\widetilde{\Delta}_{\mathbf{p}})^{-1} 
  \widetilde{\mathcal{G}}=\mathcal{H} \Delta_{\mathbf{p}} (\I - \mathcal{F} \Delta_{\mathbf{p}})^{-1} \mathcal{G} \), from which
  \eqref{eq:SumMatricesPsi} follows.
  \end{proof}
\begin{Notation}
\label{matrix_prod}
Consider the $\NX \times \NX$ matrices $\{\LPVA_i\}_{i=0}^{\NP}$ and consider a word $v \in (\mathbb{I}_0^{\NPSI})^{*}$.
Define the matrix $\LPVA_v$ as follows:
$\LPVA_{\epsilon}=I_{\NX}$, and for any $v=j_1j_2\cdots j_k$, $k > 0$, $j_1,\ldots,j_k \in \mathbb{I}_0^{\NPSI}$,
$\LPVA_v = \LPVA_{j_k}\LPVA_{j_{k-1}} \cdots \LPVA_{j_1}$.
%
\end{Notation}

 Define the map $\phi:(\mathbb{I}_1^{\NP})^{*} \rightarrow (\mathbb{I}_1^{\NP+1})^{*}$ as follows:
 $\phi(\epsilon)=\epsilon$, $\phi(i_1i_2 \cdots i_k)=(i_1+1)(i_2+1)\cdots (i_k+1)$.
 The proof of  Theorem \ref{th:FALPVtoLFT} follows from the following series of lemmas. 
 \begin{Lemma}
 \label{decomp:word}
  For any $w \in (\mathbb{I}_1^{\NP+1})^{*}$, either $w$ does not contain the symbol $1$, and there exist a unique 
 $v_1 \in (\mathbb{I}_1^{\NP})^{*}$ such that $w=\phi(v_1)$, or there exist a unique collection of
  integers $\{\alpha_i\}_{i=1}^{k}$ 
  and words $\{v_i\}_{i=1}^{k+1}$ from $(\mathbb{I}_1^{\NP})^{*}$, 
  such that 
  \begin{equation}
  \label{decomp:word:eq}
    w=\phi(v_1)1^{\alpha_1}\phi(v_2)1^{\alpha_2}\cdots \phi(v_k)1^{\alpha_k}\phi(v_{k+1}). 
   \end{equation}
   and  $|v|_i > 0$, $i=2,\ldots,k$ 
 \end{Lemma}
  The proof of Lemma \ref{decomp:word} is by induction on the length of $w$. 
  
  Define the maps $\tilde{\LPVS}_{X}$ by $\tilde{\LPVS}_{X}(\nu)=\sum_{i=1}^{\NPSI} \LPVX_i \LPVS^i(\nu)$,
  $\mathscr{X}=\mathscr{A},\mathscr{B},\mathscr{C},\mathscr{D}$.
   From the definition of $\hat{\LPVS}$ it follows that
  \begin{equation}
  \label{FALPV2LFT:lemma1:eq2} 
  \begin{split}
  \tilde{\LPVS}_{A}(\nu)={[\mathcal{H}_{x} \# \mathcal{F} \# \mathcal{G}_{u}]}_{\nu}, \quad  
  \tilde{\LPVS}_{B}(\nu)=[\mathcal{H}_{x} \# \mathcal{F} \# \mathcal{G}_u]_{\nu},  \\
 \tilde{\LPVS}_{C}(\nu)=[\mathcal{H}_{y} \# \mathcal{F} \# \mathcal{G}_x]_{\nu}, \quad 
  \tilde{\LPVS}_{D}(\nu)=[\mathcal{H}_{y} \# \mathcal{F} \# \mathcal{G}_u]_{\nu}
  \end{split}
  \end{equation}
  \begin{Lemma}
  \label{FALPV2LFT:lemma2.0}
  For any $v \in (\mathbb{I}_1^{\NP})^{*}$, $|v| > 0$,
   $[\LFTA]_{1\phi(v)1}=\tilde{\LPVS}_{A}(v)$, $[\LFTA \# \LFTB]_{\phi(v)1}=\tilde{\LPVS}_{B}(v)$, and 
   $[\LFTC \# \LFTA]_{1\phi(v)}=\tilde{\LPVS}_{C}(v)$.
  \end{Lemma}
  \begin{proof}
  We prove the first equation, the proof of the rest is similar.
   Let $v=i_1\cdots i_k$, $i_1,\ldots,i_k \in \mathbb{I}_1^{\NP}$, $k > 0$. Then 
   \[
      \begin{split}
       & [\LFTA]_{1\phi(v)1}=\LFTA_{1,i_k+1}(\LFTA_{i_k+1,i_{k-1}+1} \cdots \LFTA_{i_2+1,i_1+1}i)\LFTA_{i_1+1,1} = \\
       & =(\mathcal{H}_{x})_{i_k} (\mathcal{F}_{i_k,i_{k-1}} \cdots \mathcal{F}_{i_2,i_1}) (\mathcal{G}_x)_{i_1} = 
     [\mathcal{H}_x \# \mathcal{F} \# \mathcal{G}_x]_{v}
      \end{split}
   \]
   and the last expression equals $\tilde{\LPVS}_{\Sigma,A}(v)$, 
    where $(\mathcal{G}_x)_i, (\mathcal{G}_u)_i$ are the first $\NX$ and the last $\NU$ columns of $\mathcal{G}_i$, 
    $(\mathcal{H}_x)_i, (\mathcal{H}_y)_i$ are the first $\NX$ and the last $\NY$ \textcolor{red}{rows of $\mathcal{H}_i$.}
  \end{proof}
  \begin{Lemma}
  \label{FALPV2LFT:lemma3}
   For any $w \in (\mathbb{I}_1^{\NP+1})^{*}$ which contains at least one  occurrence of the symbol $1$, 
   there exists a finite subset $S \subseteq (\mathbb{I}_0^{\NPSI})^{*}$ which depends only $w$, such
   that $[\LFTA \# \LFTB]_{w1}$, $[\LFTC \# \LFTA]_{1w}$, $[\LFTC \# \LFTA \# \LFTB]_{1w1}$ 
   can be written as
   as a linear combination of $\{\LPVA_v \LPVB_j\}_{v \in S, j \in \mathbb{I}_0^{\NP}}$,
  $\{\LPVC_j \LPVA_v\}_{v \in S, j \in \mathbb{I}_0^{\NP}}$, $\{\LPVC_j \LPVA_v \LPVB_i\}_{v \in S, i,j \in \mathbb{I}_0^{\NP}}$,
  respectively, and the coefficients of these linear combinations depend only on $w$ and $\{\psi^l\}_{l=1}^{\NP}$.

%
%
%

  For any sequence $w \in (\mathbb{I}_1^{\NP+1})^{*}$ which contains no $1$,
  $[\LFTA \# \LFTB]_{w}$, $[\LFTC \# \LFTA]_{w}$, $[\LFTC \# \LFTA \# \LFTB]_{w}$,
  equal to 
  $[\mathcal{F} \# \mathcal{G}_u]_{\phi^{-1}(w)}$, $[\mathcal{H}_y \# \mathcal{F}]_{\phi^{-1}(w)}$, 
  $[\mathcal{H}_y \# \mathcal{F} \# \mathcal{G}_u]_{\phi^{-1}(w)}$ respectively.
  \end{Lemma}
  \begin{proof}
  \textcolor{red}{If $w$ contains at least one $1$, then
  \eqref{decomp:word:eq} holds.}
  Then  
  $[\LFTA \# \LFTB]_{w1}$
  $[\LFTC \# \LFTA]_{1w}$ and $[\LFTC \# \LFTA \# \LFTB]_{1w1}$ \textcolor{red}{are products} of matrices of the form
  $[\LFTA]_{1\phi(v_{k+1}1)}$, $[\LFTC \# \LFTA]_{\phi(v_{k+1})1}$,
  $[\LFTA]_{1\phi(v_i)1^{\alpha_i}}$, $[\LFTA \# \LFTB]_{\phi(v_{1})1}$, $i \in \mathbb{I}_1^k$.
  Note that if  $\alpha > 1$, then $[\LFTA]^{\alpha}=\LPVA_0^{\alpha-1}$, and
  $[\LFTA]_{\phi(v)1^{\alpha}}=\LPVA_0^{\alpha-1} [\LFTA]_{1\phi(v)}$.
  By Lemma \ref{FALPV2LFT:lemma2.0}, we can write $[\LFTA]_{1\phi(v)1}$ as
  $\tilde{\LPVS}_{A}(v)$, and the latter is a linear combination of $\{\LPVA_l\}_{l=1}^{\NP}$, and the coefficients of this linear 
  combination do not depend on the $\Sigma$. It then follows that 
  $\LPVA_0^{\alpha-1} [\LFTA]_{1\phi(v)}=[\LFTA]_{\phi(v)1^{\alpha}}$ is a linear combination of 
  $\{\LPVA^{\alpha-1} \LPVA_l\}_{l=1}^{\NPSI}$, and the coefficients of this linear combination do not depend on 
  the matrices of $\Sigma$.
  Similarly, from Lemma \ref{FALPV2LFT:lemma2.0}, $[\LFTA \# \LFTB]_{\phi(v)1}$ and $[\LFTC \# \LFTA]_{1\phi(v)}$
  \textcolor{red}{are linear combinations} of $\{\LPVB_l\}_{l=1}^{\NPSI}$ and $\{\LPVC_l\}_{l=1}^{\NPSI}$, respectively, and the coefficients of these linear combinations 
  do not depend on the matrices of $\Sigma$.
  Hence, \textcolor{red}{since} $[\LFTA \# \LPVB]_{w1}$, $[\LFTC \# \LFTA]_{1w}$ and $[\LFTC \# \LFTA \# \LFTB]_{1w1}$ 
  \textcolor{red}{are products of matrices}, each of which \textcolor{red}{is a} linear combination of either
  $\{\LPVB_l\}_{l=1}^{\NP}$, $\{\LPVC_l\}_{l=1}^{\NP}$, $\{ \LPVA_0^{\alpha}\LPVA_l\}_{l=1}^{\NPSI}$, $\LPVA_0^{\alpha}$, for some
  $\alpha < |w|$,  and the coefficients of these linear combinations do not depend on the matrices of $\Sigma$.
  Hence $[\LFTA]_{w}$ satisfies the statement of the theorem.
  \end{proof}
  \begin{Lemma}
  \label{FALPV2LFT:lemma4}
   There exist vectors $\mathbf{p}_1,\ldots, \mathbf{p}_{\NPSI+1} \in \mathbb{P}$ and real numbers
   $\lambda_i^l$, $i \in \mathbb{I}_1^{\NPSI+1}$, $l \in \mathbb{I}_1^{\NPSI}$ such that for all $l \in \mathbb{I}_1^{\NPSI}$, 
   $\sum_{i=1}^{\NPSI+1} \lambda_i^l=1$ and 
   and for all  $j \in \mathbb{I}_{1}^{\NPSI}$,  $\sum_{i=1}^{\NPSI+1} \lambda_i^{l} \psi_j(\mathbf{p}_i)=\delta_{j,l}$, 
  where 
  $\delta_{j,l}=1$ if $j=l$ and $\delta_{j,l}=0$ for $j \ne l$.
  \end{Lemma}
  \begin{proof}
   \textcolor{red}{Since $1, \psi_1,\ldots,\psi_{\NPSI}$ are linearly independent, }
   \textcolor{red}{the set $\mathcal{V}=\{ (\psi_1(\mathbf{p}),\ldots, \psi_{\NPSI}(\mathbf{p}))^T \mid  \mathbf{p} \in \mathbb{P}\}$}
   contains an affine basis of $\mathbb{R}^{\NPSI}$. 
   Indeed, let $r$ be the dimension of the affine hull  of \textcolor{red}{$\mathcal{V}$.}
   If $r < \NPSI$, then by \cite[Corollary 1.4.2]{WebsterBook}, the affine hull of 
   $\mathcal{V}$
   belongs to at least one hyperplane, i.e., there exists 
   $n=\begin{bmatrix} n_1 & \ldots & n_{\NPSI} \end{bmatrix}^T \in \mathbb{R}^{\NPSI}$ 
   and $c \in \mathbb{R}$ such that $n \ne 0$ and for any
   $x=(\psi_1(\mathbf{p}),\ldots, \psi_{\NPSI}(\mathbf{p}))^T$, $\mathbf{p} \in \mathbb{P}$, 
   $n^Tx=c$, which is equivalent to $\sum_{i=1}^{\NPSI} n_i \psi_i - c \cdot 1 =0$,
   and the latter contradicts to linear independence of $1, \psi_1,\ldots, \psi_{\NPSI}$, since $n \ne 0$.  
   Hence, $r=\NPSI$.
   If the affine hull of  \textcolor{red}{$\mathcal{V}$}
   is of dimension $r=\NPSI$, then by \cite[page 15]{WebsterBook},  
   \textcolor{red}{$\mathcal{V}$ contains an affine basis $b_1,\ldots, b_{\NPSI+1}$}
   of $\mathbb{R}^{\NPSI}$. 
   This implies that there exists $\mathbf{p}_1,\ldots,\mathbf{p}_l$ such that 
    $b_i=\begin{bmatrix} \psi_1(\mathbf{p}_i) & \psi_2(\mathbf{p}_i) & \ldots & \psi_{\NPSI}(\mathbf{p}_i) \end{bmatrix}^T$. Since 
    $b_1,\ldots, b_{\NPSI+1}$ is an affine basis of $\mathbb{R}^{\NPSI}$, for any $b \in \mathbb{R}^{\NPSI}$,$b$ can be expressed as an affine combination of $b_1,\ldots,b_{\NPSI+1}$, i.e. there exist
 real numbers $\lambda_1,\ldots,\lambda_{\NPSI+1}$ such that $\sum_{l=1}^{\NPSI+1} \lambda_l=1$ and
 $b=\sum_{l=1}^{\NPSI+1} \lambda_i b_i$. By choosing $b$ as $e_1,\ldots,e_{\NPSI}$, where 
 $e_l$ is the $m$th standard basis vector of $\mathbb{R}^{\NPSI}$ (the $j$th element of $e_l$ equals $1$, all the other elements are zero, i.e., 
 the $j$th element of $e_l$ equals $\delta_{j,l}$), it follows that for any $l \in \mathbb{I}_1^{\NPSI}$ there exist real numbers $\lambda_i^l$, 
 $i \in \mathbb{I}_1^{\NPSI+1}$ such that for all 
$\sum_{i=1}^{\NPSI+1} \lambda_i^l =1$, and
$\sum_{i=1}^{\NPSI+1} \lambda_i^{l} b_i=e_l$. By taking the $j$th entry of both sides of the latter equation, 
the statement follows.
   \end{proof}
  \begin{Lemma}
  \label{FALPV2LFT:lemma5}
   For any $v \in (\mathbb{I}_0^{\NPSI})^{*}$, $i,j \in \mathbb{I}_0^{\NPSI}$, 
   the matrices $\LPVA_v\LPVB_j$, $\LPVC_i \LPVA_v$, $\LPVC_i \LPVA_v \LPVB_j$ 
   are limits of 
   linear combinations\footnote{\textcolor{red}{
   by limits of linear combinations of some family of matrices we mean limits $\lim_{k \rightarrow \infty}  Z_k$,
   where for each $k$, $Z_k$ is a linear combination of some matrices of the family.}} 
   of 
   $\{[\LFTA \# \LFTB]_{w1}\}_{w \in (\mathbb{I}_1^{\NP+1})^{*}}$,
    $\{[\LFTC \# \LFTA]_{1w}\}_{w \in (\mathbb{I}_1^{\NP+1})^{*}}$,
    $\{[\LFTC \# \LFTA \# \LFTB]_{1w1}\}_{w \in (\mathbb{I}_1^{\NP+1})^{*}}$, respectively, and the coefficients of those linear 
    combinations are independent of  $\Sigma$ and $\psi$. 
%
%
  \end{Lemma}
  \begin{proof}
    Each matrix $\mathcal{Z}=\LPVA_i,\LPVB_i,\LPVC_i,\LPVD_i$, 
    is a  limit 
    of linear combinations of the products of matrices of the form
    $\{[\LFTA]_{1\phi(\nu)1}\}_{\nu \in (\mathbb{I}_1^{\NP})^{*}}$,
    $\{[\LFTA \# \LFTB]_{\phi(\nu)1}\}_{\nu \in (\mathbb{I}_1^{\NP})^{*}}$, 
    $\{[\LFTC \# \LFTA]_{1\phi(\nu)}\}_{\nu \in (\mathbb{I}_1^{\NP})^{*}},
    \{[\LFTC \# \LFTA \# \LFTB]_{\phi(\nu)}\}_{\nu \in (\mathbb{I}_1^{\NP})^{*}}$, 
    respectively, for all $i \in \mathbb{I}_0^{\NP}$. The coefficients of these linear combinations 
    do not depend on the matrices of $\Sigma$ and $\mathcal{M}(\Sigma,\psi)$.
    To this end, note that $\LPVA_0=[\LFTA]_{11}$ and $\tilde{\LPVS}_{\Sigma,A}(\nu)=[\LFTA]_{1\phi(\nu)1}$ by Lemma \ref{FALPV2LFT:lemma2.0}.  By
    Lemma \ref{FALPV2LFT:lemma4},
    for all $l \in \mathbb{I}_1^{\NPSI}$,
    \( \LPVA_l=\sum_{j=1}^{\NPSI+1}  \lambda_i^l \LPVA(\mathbf{p}_i) - \textcolor{red}{\LPVA_0}=
        \sum_{j=1}^{\NPSI+1} \sum_{\nu \in (\mathbb{I}_1^{\NP})^{*}} \lambda_i^l \tilde{\LPVS}_{\Sigma,A}(\nu) \LPVP(\mathbf{p}_i,\nu)
     \)
     where $\LPVP(\mathbf{p}_i,\nu)$ is a suitable monomial in $\mathbf{p}_i$ determined by $\nu$.
     Notice that $\tilde{\LPVS}_{\Sigma,A}(\nu)=[\LFTA]_{1\phi(\nu)1}]$.
     Hence, the statement for $\LPVA_i$ follows. 
   The proof for $\LPVB_i$ and $\LPVC_i$ is similar.
   Since $\LPVA_{v}$ is a product of $\LPVA_{l}$, $l\in \mathbb{I}_0^{\NPSI}$, it then follows that
   $\LPVA_{v}$ is the limit of linear combinations of
   products of $\{[\LFTA]_{1\phi(\nu)1}\}_{\nu \in (\mathbb{I}_1^{\NP})^{*}}$ of length $|v|$. 
   The latter products corresponds to elements of $\{ [\LFTA]_{1w1} \}\}_{w \in (\mathbb{I}_1^{\NP+1})^{*}}$.
   Since $\LPVB_{j}$ is a limit of suitable linear combinations of 
   $\{[\LFTA  \# \LFTB]_{\phi(\nu)1}\}_{\nu \in (\mathbb{I}_1^{\NP})^{*}}$, hence,
   $\LPVA_v\LPVB_j$ is the limit of linear combinations of products of elements of
   $\{[\LFTA \# \LFTB]_{w1}\}_{w \in (\mathbb{I}_1^{\NP+1})^{*}}$ and 
    $\{[\LFTA]_{1\phi(\nu)1}\}_{\nu \in (\mathbb{I}_1^{\NP})^{*}}$. The latter products are elements of
    $\{ [\LFTA \# \LFTB]_{w1}\}_{w \in (\mathbb{I}_1^{\NP+1})^{*}}$.
    The proofs for the statements for  $\LPVC_i \LPVA_v$ and $\LPVC_i \LPVA_v \LPVB_j$ are similar  
  \end{proof}
 \begin{proof}[Proof of Theorem \ref{th:FALPVtoLFT}]
 \textbf{Part \ref{th:FALPVtoLFT:part0}}
  \textcolor{red}{For a solution  $(x,y,u,p)$ of $\Sigma$}, set 
  $\tilde{z}(t)\!\!=\!\!(\I\!-\!\mathcal{F}\Delta_{p(t)})^{-1}\mathcal{G}\begin{bmatrix} x^T(t) & u^T(t) \end{bmatrix}$.
  By Lemma \ref{lem:rational}, the inverse used in the definition of
  $\tilde{z}$ is well-defined. From \eqref{eq:SumMatricesPsi}, 
  if $w(t)=\Delta_{p(t)}\tilde{z}(t)$, then 
  $\mathcal{H}_x w(t)=\sum_{l=1}^{\NPSI} (A_lx(t)+B_lu(t))\psi_l(p(t))$ and 
  $\mathcal{H}_y w(t)=\sum_{l=1}^{\NPSI} (C_lx(t)+D_lu(t))\psi_l(p(t))$.
  Hence, 
  \textcolor{red}{\eqref{sol:LFT} holds}, i.e.,
  $(z=(x^T,\tilde{z}^T)^T,y,u)$ is a solution of 
  $\mathcal{M}(\Sigma)$ for $\Delta(p)$.

 \textbf{Part \ref{th:FALPVtoLFT:part1}}
   \textcolor{red}{Assume that $\hat{\Sigma}$ is of the form \eqref{equ:FALPVSystem} with 
  the matrices $\hat{\LPVA}_i, \hat{\LPVB}_i, \hat{\LPVC}_i, \hat{\LPVD}_i$ instead of $\LPVA_i, \LPVB_i, \LPVC_i, \LPVD_i$, respectively,
   and $\mathcal{M}(\hat{\Sigma},\psi)$
   is as in \eqref{psi_LFT} with $\hat{\NX}, \hat{\mathcal{F}},\hat{\mathcal{G}},\hat{\mathcal{H}}, \hat{\mathcal{D}}$
   instead of $\NX,\mathcal{F},\mathcal{G},\mathcal{H},\mathcal{D}$, respectively.}
    \textcolor{red}{If} $\Sigma$ and $\hat{\Sigma}$ are input-output equivalent, 
    \textcolor{red}{then $\LPVC_i \LPVA_v \LPVB_j=\hat{\LPVC}_i \hat{\LPVA}_v \hat{\LPVB}_j$, $\LPVD_i=\hat{\LPVD}_i$,}
    for all $v \in (\mathbb{I}_0^{\NPSI})^{*}$, 
    $i,j \in \mathbb{I}_0^{\NPSI}$. 
    \textcolor{red}{We show that}
    $Y_{\mathcal{M}(\Sigma)}=Y_{\mathcal{M}(\hat{\Sigma})}$, \textcolor{red}{i.e.,}
    for all  $w \in (\mathbb{I}_1^{\NP+1})^{*}$, 
    $[\LFTC \# \LFTA \# \LFTB]_w=[\hat{\LFTC} \# \hat{\LFTA} \# \hat{\LFTB}]_w$. 
    \textcolor{red}{If} $w$ contains at least one symbol $1$, then by Lemma \ref{FALPV2LFT:lemma3},
    $[\LFTC \# \LFTA \# \LFTB]_w$ and $[\hat{\LFTC} \# \hat{\LFTA} \# \hat{\LFTB}]_w$ are 
    linear combinations of $\{\LPVC_i \LPVA_v \LPVB_j\}_{v \in (\mathbb{I}_0^{\NPSI})^{*}}$ 
    and $\{\hat{\LPVC}_i \hat{\LPVA}_v \hat{\LPVB}_j\}_{v \in (\mathbb{I}_0^{\NPSI})^{*},i,j \in \mathbb{I}_0^{\NP}}$, respectively,
    the coefficients corresponding of $\LPVC_i \LPVA_v \LPVB_j$ and $\hat{\LPVC}_i \hat{\LPVA}_v \hat{\LPVB}_j$
    in those linear combinations are the same, as they do not depend on $\Sigma$. 
    Hence, as  $\LPVC_i \LPVA_v \LPVB_j=\hat{\LPVC}_i \hat{\LPVA}_v \hat{\LPVB}_j$, the linear combinations are the same
    and hence $[\LFTC \# \LFTA \# \LFTB]_w=[\hat{\LFTC} \# \hat{\LFTA} \# \hat{\LFTB}]_w$.
    If $w$ contains no $1$, then by Lemma \ref{FALPV2LFT:lemma3}, 
    $[\LFTC \# \LFTA \# \LFTB]_w$ and $[\hat{\LFTC} \# \hat{\LFTA} \# \hat{\LFTB}]_w$ are equal to 
    $[\mathcal{H}_y \# \mathcal{F} \# \mathcal{G}_u]_{\phi^{-1}(w)}=\sum_{l=1}^{\NPSI} \LPVD_l \LPVS^{l}(\phi^{-1}(w))$ and 
    $[\hat{\mathcal{H}}_y \# \hat{\mathcal{F}} \# \hat{\mathcal{G}}_u]_{\phi^{-1}(w)}=\sum_{l=1}^{\NPSI} \hat{\LPVD_l} \LPVS^{l}(\phi^{-1}(w)$
    and since $\hat{D}_l=D_l$ for all $l \in \mathbb{I}_1^{\NPSI}$, the last two quantities are equal.

   Conversely, assume that $\mathcal{M}(\Sigma)$ and $\mathcal{M}(\hat{\Sigma})$ are formally input-output equivalent. 
   Then 
    $[\LFTC \# \LFTA \# \LFTB]_w=[\hat{\LFTC} \# \hat{\LFTA} \# \hat{\LFTB}]_w$ for all $w \in (\mathbb{I}_1^{\NP+1})^{*}$, $|w| > 0$.
    and $\LPVD_0=\LFTD=\hat{\LFTD}=\hat{\LPVD}_0$.
    
   By Lemma \ref{FALPV2LFT:lemma5}, 
   for any $v \in (\mathbb{I}_0^{\NPSI})^{*}$, $i,j \in \mathbb{I}_0^{\NPSI}$, 
    $\LPVC_i \LPVA_v\LPVB_j$ and $\hat{\LPVC}_i \hat{\LPVA}_v\hat{\LPVB}_j$
    are both limits of linear combinations of  
    $\{[\LFTC \# \LFTA \# \LFTB]_{1w1} \}_{w \in (\mathbb{I}_1^{\NP+1})^{*}}$
    and $\{[\hat{\LFTC} \# \hat{\LFTA} \# \hat{\LFTB}]_{1w1}\}_{w \in (\mathbb{I}_1^{\NP+1})^{*}}$
    and the coefficients of these linear combinations  are the same. Hence, 
    as  $[\LFTC \# \LFTA \# \LFTB]_w=[\hat{\LFTC} \# \hat{\LFTA} \# \hat{\LFTB}]_w$,  $w \in (\mathbb{I}_1^{\NP+1})^{*}$, 
    it follows that the linear combinations and thus their limits are the same. 
    Hence, $\LPVC_i \LPVA_v\LPVB_j=\hat{\LPVC}_i \hat{\LPVA}_v\hat{\LPVB}_j$ for all $v \in (\mathbb{I}_0^{\NPSI})^{*}$, $i,j \in \mathbb{I}_0^{\NPSI}$,
    i.e. $\Sigma$ and $\hat{\Sigma}$ are input-output equivalent.

   \textbf{Part \ref{th:FALPVtoLFT:part2}}
   Assume $\Sigma$ is minimal and $\mathcal{M}(\Sigma,\psi)$ is minimal. 
   \textcolor{red}{We show that $\mathcal{M}(\Sigma)$ is reachable.
   Observability can be  shown by duality.}
   Assume that $\mathcal{M}(\Sigma)$ is not reachable. Then there exists a vector $v \ne 0$ and  $i \in \mathbb{I}_0^{\NP}$
   such that $v^T[\LFTA \# \LFTB]_{wi}=0$ for all $w \in (\mathbb{I}_1^{\NP+1})^{*}$.
   Assume that $i \ne 1$ and 
   note that $[\LFTA \# \LFTB]_{wi}=[\mathcal{F} \# \mathcal{G}_x]_{\phi^{-1}(wi)}$ if $wi$ does not contain $1$.
   Then for any $i \in \mathbb{I}_1^{\NP}$, $v_i^T[\mathcal{F} \# \mathcal{G}_x]_{(i-1)s}=0$ for all $s \in (\mathbb{I}_1^{\NP})^{*}$.
   Hence, by reachability of $\mathcal{M}(\Sigma,\psi)$, $v_i=0$, $i \in \mathbb{I}_1^{\NP}$,
   which is a contradiction.
   If $i=1$, then $v^T[\LFTA \# \LFTB]_{w1}=0$ for all $w \in (\mathbb{I}_1^{\NP+1})^{*}$. 
   Since $\LPVA_s \LPVB_j$, $s \in (\mathbb{I}_1^{\NPSI})^{*}$ is a limit of linear combinations of
   $\{[\LFTA \# \LFTB]_{w1}\}_{w \in (\mathbb{I}_1^{\NP+1})^{*}}$, then $v^T \LPVA_s \LPVB_j=0$ for all $j \in \mathbb{I}_0^{\NPSI}$,
   and hence by reachability of $\Sigma$, $v=0$, which is a contradiction.     
 \end{proof} 
\begin{proof}[Proof of Lemma \ref{rem:comp1:lem0}]
As it was mentioned in  proof of Lemma \ref{lem:rational},
for any $w \in (\mathbb{I}_1^{\NP})^{*}$, $i \in (\mathbb{I}_1^{\NP})$
\[ [\widetilde{\mathcal{F}} \# \widetilde{\mathcal{G}} ]_{wi}=[F \# G]_{wi} \otimes \I_{\NY+\NU} \]
We will show that $\{[\widetilde{\mathcal{F}} \# \widetilde{\mathcal{G}} ]_{wi}\}_{w  \in (\mathbb{I}_1^{\NP})^{*}}$ spans $\mathbb{R}^{n_i'}$.
Then the claim of Lemma \ref{rem:comp1:lem0} follows from \cite[Theorem 20]{Beck1999}. 
To this end, we show that for every $v \in \mathbb{R}^{n_i'}$, if $\forall w  \in (\mathbb{I}_1^{\NP})^{*}: v^T  [\widetilde{\mathcal{F}} \# \widetilde{\mathcal{G}}]_{wi}=0$,
then $v=0$. Note that by \cite[Section 12.3.4]{golub2013matrix}, there exists a matrix $X$ such that $v=\mathrm{vect}(X)$, where 
$\mathrm{vect}$ is the vectorization operator defined in  \cite[Section 12.3.4]{golub2013matrix}, and
\cite[Section 12.3.4, equation 12.3.11, Section 12.3.1]{golub2013matrix},
$(v^T [\widetilde{\mathcal{F}} \# \widetilde{\mathcal{G}}]_{wi})^T=(\I_{\NY+\NU} \otimes [F \# G]_{wi}^T)\mathrm{vect}(X)=[F \# G]^T_{wi}X=0$.
In other words, $0=([F \# G]^T_{wi}X)^T=X[F \# G]_{wi}$. By minimality of $M$ it follows that
the columns of $\{[F \# G]_{wi}\}_{w  \in (\mathbb{I}_1^{\NP})^{*}}$ span $\mathbb{R}^{n_i}$ and hence $X=0$ and thus $v=0$.
\end{proof}
\begin{proof}[Proof of Lemma \ref{rem:comp1:lem}]
	By Lemma \ref{lem:rational} $\widetilde{\mathcal{M}}(\Sigma,\psi)$ satisfies the reachability rank condition of \cite[Theorem 20]{Beck1999}.
	In order to show minimality, it is sufficient to show that  $\widetilde{\mathcal{M}}(\Sigma,\psi)$ satisfies the observability rank condition of \cite[Theorem 20]{Beck1999}
	To this end, we show that if  $v \in \mathbb{R}^{n_i'}$ is such that 
	$\forall w  (\mathbb{I}_1^{\NP})^{*}:[\widetilde{\mathcal{H}} \# \widetilde{\mathcal{F}}]_{iw}v=0$, then $v=0$. 
	Let  $v \in \mathbb{R}^{n_i'}$ is such that 
	$\forall w  \in (\mathbb{I}_1^{\NP})^{*}:[\widetilde{\mathcal{H}} \# \widetilde{\mathcal{F}}]_{iw}v=0$.
         Note that by \cite[Section 12.3.4]{golub2013matrix}, there exists a matrix $X$ such that $v=\mathrm{vect}(X)$, where 
$\mathrm{vect}$ is the vectorization operator defined in  \cite[Section 12.3.4]{golub2013matrix}.
	As it was mentioned in  proof of Lemma \ref{lem:rational},
	\[ [\widetilde{\mathcal{H}} \# \widetilde{\mathcal{F}}]_{i w}=[H \# F]_{i w} \otimes \begin{bmatrix} \LPVA_1 & \LPVB_1 \\ \LPVC_1 & \LPVD_1 \end{bmatrix} \]
Hence, \cite[Section 12.3.4, equation 12.3.11, Section 12.3.1]{golub2013matrix},
	\begin{align*}
		& [\widetilde{\mathcal{H}} \# \widetilde{\mathcal{F}}]_{iw}v=([H \# F]_{iw} \otimes \begin{bmatrix} \LPVA_1 & \LPVB_1 \\ \LPVC_1 & \LPVD_1 \end{bmatrix}) \mathrm{vect}(X) \\
			&   =\begin{bmatrix} \LPVA_1 & \LPVB_1 \\ \LPVC_1 & \LPVD_1 \end{bmatrix} X ([H \# F]_{iw})^T = 0
	\end{align*}
	If  $\begin{bmatrix} \LPVA_1 & \LPVB_1 \\ \LPVC_1 & \LPVD_1 \end{bmatrix}$ is full column rank, then it follows that  $X ([H \# F]_{iw})^T=0$, i.e.,
		$[H \# F]_{iw} X^T=0$. Since the latter holds for all  $w  \in  (\mathbb{I}_1^{\NP})^{*}$, by minimality of $M$ and the observability rank condition of 
        \cite[Theorem 20]{Beck1999}, $X^T=0$, which implies $v=0$.
\end{proof}
\begin{proof}[Proof of Lemma \ref{rem:comp1:lem1}]
	Similary to the proof of Lemma \ref{lem:rational},  using the properties of the Kronecker product  \cite[Section 12.3.1]{golub2013matrix}, 
	it follows that for any  $w \in (\mathbb{I}_1^{\NP})^{*}$, $i,j \in (\mathbb{I}_1^{\NP})$
	\begin{align*}
		& [\mathcal{H} \# \mathcal{F}]_{iw}=[H \# F]_{iw} \otimes L, \quad 
		[\mathcal{F} \# \mathcal{G}]_{wj}=[F \# G]_{wj} \otimes R, \\
		&[\mathcal{H} \# \mathcal{F} \# \mathcal{G}]_{iwj}=[H \# F \# G]_{iwj} \otimes RL=\\
		& \mathcal{S}(iwj) \begin{bmatrix} \LPVA_1 & \LPVB_1 \\ \LPVC_1 & \LPVD_1 \end{bmatrix}
			=\tilde{\mathcal{S}}_{\Sigma}(iwj)
	\end{align*}
	That is, $\mathcal{M}(\Sigma,\psi)$ is formaly input-output equivalent with $\widetilde{\mathcal{M}}(\Sigma,\psi)$.
	Moreover, in the same manner as in the proof of Lemma \ref{rem:comp1:lem0}-\ref{rem:comp1:lem}, it can be shown that
	$\mathcal{M}(\Sigma,\psi)$ satisfies the controllability and observability rank conditions of \cite[Theorem 20]{Beck1999}. 

	Indeed, in order to show that the controllability rank condition  of \cite[Theorem 20]{Beck1999} holds, we show that
	 $v \in \mathbb{R}^{n_i'm}$, if $\forall w  \in (\mathbb{I}_1^{\NP})^{*}: v^T  [\mathcal{F} \# \mathcal{G}]_{wi}=0$,
then $v=0$. Note that by \cite[Section 12.3.4]{golub2013matrix}, there exists a matrix $X$ such that $v=\mathrm{vect}(X)$, where 
$\mathrm{vect}$ is the vectorization operator defined in  \cite[Section 12.3.4]{golub2013matrix}, and
\cite[Section 12.3.4, equation 12.3.11, Section 12.3.1]{golub2013matrix},
	$(v^T [\mathcal{F} \# \mathcal{G}]_{wi})^T=(R^T \otimes [F \# G]_{wi}^T)\mathrm{vect}(X)=[F \# G]^T_{wi}XR=0$.
	Since $R$ is full row rank, $[F \# G]^T_{wi}XR=0$ implies $[F \# G]^T_{wi}X=0$.
In other words, $0=([F \# G]^T_{wi}X)^T=X[F \# G]_{wi}$. By minimality of $M$ it follows that
the columns of $\{[F \# G]_{wi}\}_{w  \in (\mathbb{I}_1^{\NP})^{*}}$ span $\mathbb{R}^{n_i}$ and hence $X=0$ and thus $v=0$.

 Similarly, in order to that the observability rank condition  of \cite[Theorem 20]{Beck1999} holds, 
 we show that if $v \in \mathbb{R}^{n_i m'}$ is such that 
	$\forall w  (\mathbb{I}_1^{\NP})^{*}:[\mathcal{H} \# \mathcal{F}]_{iw}v=0$, then $v=0$. 
	Let  $v \in \mathbb{R}^{n_i'm }$ is such that 
	$\forall w  \in (\mathbb{I}_1^{\NP})^{*}:[\mathcal{H} \# \mathcal{F}]_{iw}v=0$.
         Note that by \cite[Section 12.3.4]{golub2013matrix}, there exists a matrix $X$ such that $v=\mathrm{vect}(X)$, where 
$\mathrm{vect}$ is the vectorization operator defined in  \cite[Section 12.3.4]{golub2013matrix}.
	As it was mentioned in  proof of Lemma \ref{lem:rational},
	\[ [\mathcal{H} \# \mathcal{F}]_{i w}=[H \# F]_{i w} \otimes L \]
Hence, \cite[Section 12.3.4, equation 12.3.11, Section 12.3.1]{golub2013matrix},
	\begin{align*}
		& [\mathcal{H} \# \mathcal{F}]_{iw}v=([H \# F]_{iw} \otimes L) \mathrm{vect}(X) = \\
			&   L X ([H \# F]_{iw})^T = 0
	\end{align*}
	As $L$ is full column rank, then it follows that  $X ([H \# F]_{iw})^T=0$, i.e.,
		$[H \# F]_{iw} X^T=0$. Since the latter holds for all  $w  \in  (\mathbb{I}_1^{\NP})^{*}$, by minimality of $M$ and the observability rank condition of 
        \cite[Theorem 20]{Beck1999}, $X^T=0$, which implies $v=0$.

\end{proof}

\color{black}

\

\section{Conclusions}\label{sec:conclusion}
We presented a transformation of  LPV models with functional affine dependence 
on the scheduling into LFTs with a linear dependence on the scheduling, which preserves both input-output equivalence and minimality. 
Future work will focus on the algorithmic aspects of this transformation and its applications in control and system identification.
\bibliographystyle{IEEEtran}
\end{document}